\documentclass{article}

\usepackage[final,nonatbib]{opt_template}
\usepackage{amsmath,amssymb}
\usepackage{caption}
\usepackage{xcolor}
\usepackage{subcaption}
\usepackage{footnote}
\makesavenoteenv{tabular}
\makesavenoteenv{table}
\usepackage{mathrsfs}
\usepackage{enumerate}
\usepackage{graphicx}
\graphicspath{{figs/}}
\usepackage{algorithm}
\usepackage{algpseudocode}

\usepackage[utf8]{inputenc} 
\usepackage[T1]{fontenc}    
\usepackage{hyperref}       
\usepackage{url}            
\usepackage{booktabs}       
\usepackage{amsfonts}       
\usepackage{nicefrac}       
\usepackage{microtype}      

\newcommand{\citet}{\cite}

\newtheorem{remark}{Remark}

\newtheorem{theorem}{Theorem}[section]
\newtheorem{lemma}[theorem]{Lemma}
\newtheorem{corollary}[theorem]{Corollary}

\newcommand{\lb}{\left(}
\newcommand{\rb}{\right)}
\newcommand{\eps}{\epsilon}

\newcommand{\td}{\tilde}
\newcommand{\E}{\mathbb{E}}
\newcommand{\R}{\mathbb{R}}
\newcommand{\I}{\mathcal{I}}

\newcommand{\la}{\langle}
\newcommand{\ra}{\rangle}
\renewcommand{\H}{\mathcal{H}}
\newcommand{\sI}{\td{\mathcal{I}}}
\newcommand{\comp}{C_{\mathrm{comp}}}

\newcommand{\xj}{x^{(j)}}
\newcommand{\Lj}{L^{(j)}}
\newcommand{\nuj}{\nu^{(j)}}
\newcommand{\etaj}{\eta_{j}}
\newcommand{\Bj}{B_{j}}
\newcommand{\bj}{b_{j}}

\newcommand{\ej}{e_{j}}
\newcommand{\Ej}{\E_{j}}

\newcommand{\Nj}{N_{j}}
\newcommand{\Mj}{M^{(j)}}
\newcommand{\thetaj}{\theta_{j}}
\def \endprf{\hfill {\vrule height6pt width6pt depth0pt}\medskip}
\newenvironment{proof}{\noindent {\bf Proof} }{\endprf\par}

\DeclareMathOperator*{\argmin}{arg\,min}

\allowdisplaybreaks
\title{Nonconvex Finite-Sum Optimization Via SCSG Methods}

\author{
  Lihua Lei \\
  UC Berkeley\\
  \texttt{lihua.lei@berkeley.edu} \\
  \And
  Cheng Ju \\
  UC Berkeley\\
  \texttt{cju@berkeley.edu} \\
  \And
  Jianbo Chen\\
  UC Berkeley\\
  \texttt{jianbochen@berkeley.edu} \\
  \And
  Michael I. Jordan \\
  UC Berkeley\\
  \texttt{jordan@stat.berkeley.edu} \\
}
\allowdisplaybreaks
\begin{document}
\maketitle

\begin{abstract}
We develop a class of algorithms, as variants of the \emph{stochastically controlled stochastic gradient} (SCSG) methods \cite{SCSG}, for the smooth non-convex finite-sum optimization problem. Assuming the smoothness of each component, the complexity of SCSG to reach a stationary point with $\E \|\nabla f(x)\|^{2}\le \eps$ is $O\lb\min\{\eps^{-5/3}, \eps^{-1}n^{2/3}\}\rb$, which strictly outperforms the stochastic gradient descent. Moreover, SCSG is never worse than the state-of-the-art methods based on variance reduction and it significantly outperforms them when the target accuracy is low. A similar acceleration is also achieved when the functions satisfy the Polyak-Lojasiewicz condition. Empirical experiments demonstrate that SCSG outperforms stochastic gradient methods on training multi-layers neural networks in terms of both training and validation loss.
\end{abstract}

\section{Introduction}\label{sec:intro}
We study smooth non-convex finite-sum optimization problems of the form
\begin{equation}
  \label{eq:obj}
  \min_{x\in \R^{d}}f(x) = \frac{1}{n}\sum_{i=1}^{n}f_{i}(x)
\end{equation}
where each component $f_{i}(x)$ is possibly non-convex with a Lipschitz gradient. This generic form captures numerous statistical learning problems, ranging from generalized linear models \cite{GLM} to deep neural networks \cite{lecun15}.
 

In contrast to the convex case, the non-convex case is comparatively under-studied. Early work focused on the asymptotic performance of algorithms~\cite{gaivoronski94, bertsekas97, tseng98}, with non-asymptotic complexity bounds emerging more recently \citet{nesterov04}. In recent years, complexity results have been derived for both gradient methods \citet{ghadimi16acc, agarwal16, carmon16, carmon17} and stochastic gradient methods \citet{ghadimi13, ghadimi16acc, Zhu15, Zhu16, reddi16svrg, reddi16saga, natasha}. 
 Unlike in the convex case, in the non-convex case one can not expect a gradient-based algorithm to converge to the global minimum if only smoothness is assumed. As a consequence, instead of measuring function-value suboptimality $\E f(x) - \inf_{x}f(x)$ as in the convex case, convergence is generally measured in terms of the squared norm of the gradient; i.e., $\E \|\nabla f(x)\|^{2}$.  We summarize the best achievable rates \footnote{It is also common to use $\E \|\nabla f(x)\|$ to measure convergence; see, e.g. \citet{agarwal16, carmon16, carmon17, natasha}. Our results can be readily transferred to this alternative measure by using Cauchy-Schwartz inequality, $\E \|\nabla f(x)\| \le \sqrt{\E \|\nabla f(x)\|^{2}}$,  although not vice versa. The rates under this alternative can be made comparable to ours by replacing $\eps$ by $\sqrt{\eps}$.} in Table \ref{tab:summary}. We also list the rates for Polyak-Lojasiewicz (P-L) functions, which will be defined in Section \ref{sec:notation}. The accuracy for minimizing P-L functions is measured by $\E f(x) - \inf_{x}f(x)$. 

\begin{table}
  \centering
  \caption{Computation complexity of gradient methods and stochastic gradient methods for the finite-sum non-convex optimization problem \eqref{eq:obj}. The second and third columns summarize the rates in the smooth and P-L cases respectively. $\mu$ is the P-L constant and $\H^{*}$ is the variance of a stochastic gradient. These quantities are defined in Section \ref{sec:notation}.  The final column gives additional required assumptions beyond smoothness or the P-L condition. The symbol $\wedge$ denotes a minimum and $\td{O}(\cdot)$ is the usual Landau big-O notation with logarithmic terms hidden.}\label{tab:summary}
  \begin{tabular}{c|c|c|l}
    \toprule
     & Smooth & Polyak-Lojasiewicz & additional cond.\\
    \midrule
    \multicolumn{2}{l}{Gradient Methods} \\
    \midrule
    GD & $O\lb\frac{n}{\eps}\rb$ \cite{nesterov04, ghadimi16acc}& $\td{O}\lb\frac{n}{\mu}\rb$ \citet{polyak63, karimi16} & -\\
    Best achievable & $\td{O}\lb \frac{n}{\eps^{7/8}}\rb$ \cite{carmon17} & - & smooth gradient \\
    & $\td{O}\lb \frac{n}{\eps^{5/6}}\rb$ \cite{carmon17} & - & smooth Hessian\\
    \midrule
    \multicolumn{2}{l}{Stochastic Gradient Methods} \\
    \midrule
   SGD & $O\lb\frac{1}{\eps^{2}}\rb$ \cite{nesterov04, reddi16svrg} & $O\lb\frac{1}{\mu^{2}\eps}\rb$ \cite{karimi16} & $\H^{*} = O(1)$ \\
   Best achievable  & $O\lb n + \frac{n^{2/3}}{\eps}\rb$\cite{reddi16svrg, reddi16saga} & $\td{O}\lb n + \frac{n^{2/3}}{\mu}\rb$\cite{reddi16svrg, reddi16saga} & - \\
   {\color{red}SCSG} & ${\color{red}\td{O}\lb\frac{1}{\eps^{5/3}}\wedge \frac{n^{2/3}}{\eps}\rb}$ & ${\color{red}\td{O}\lb(\frac{1}{\mu\eps}\wedge n) + \frac{1}{\mu}(\frac{1}{\mu\eps}\wedge n)^{2/3}\rb}$ & $\H^{*} = O(1)$ \\
    \bottomrule
  \end{tabular}
\end{table}

As in the convex case, gradient methods have better dependence on $\eps$ in the non-convex case but worse dependence on $n$.  This is due to the requirement of computing a full gradient. Comparing the complexity of SGD and the best achievable rate for stochastic gradient methods, achieved via variance-reduction methods, the dependence on $\eps$ is significantly improved in the latter case. However, unless $\eps <\!\!< n^{-1/2}$, SGD has similar or even better theoretical complexity than gradient methods and existing variance-reduction methods. In practice, it is often the case that $n$ is very large ($10^5\sim 10^9$) while the target accuracy is moderate ($10^{-1}\sim 10^{-3}$). In this case, SGD has a meaningful advantage over other methods, deriving from the fact that it does not require a full gradient computation.  This motivates the following research question: Is there an algorithm that

\begin{itemize}
\item achieves/beats the theoretical complexity of SGD in the regime of modest target accuracy;
\item and achieves/beats the theoretical complexity of existing variance-reduction methods in the regime of high target accuracy?
\end{itemize}

The question has been partially answered in the convex case by \cite{SCSG} in their formulation of the \emph{stochastically controlled stochastic gradient} (SCSG) methods. When the target accuracy is low, SCSG has the same $O\lb\eps^{-2}\rb$ rate as SGD but with a much smaller data-dependent constant factor (which does not even require bounded gradients). When the target accuracy is high, SCSG achieves the same rate as the best non-accelerated methods, $O(\frac{n}{\eps})$. Despite the gap between this and the optimal rate, SCSG is the first known algorithm that provably achieves the desired performance in both regimes.

In this paper, we show how to generalize SCSG to the non-convex setting and, surprisingly, provide a completely affirmative answer to the question raised above. Even though we only assume smoothness of each component, we show that SCSG is always $O\lb\eps^{-1/3}\rb$ faster than SGD and is never worse than recently developed variance-reduction methods. 
When $\eps >\!\!> \frac{1}{n}$, SCSG is at least $O((\eps n)^{2/3})$ faster than the best variance-reduction algorithm. Comparing with gradient methods, SCSG has a better convergence rate provided $\eps >\!\!> n^{-6/5}$, which is the common setting in practice. Interestingly, there is a parallel to recent advances in gradient methods; \citet{carmon17} improved the classical $O(\eps^{-1})$ rate of gradient descent to $O(\eps^{-5/6})$; this parallels the improvement of SCSG over SGD from $O(\eps^{-2})$ to $O(\eps^{-5/3})$. 

Beyond the theoretical advantages of SCSG, we also show that SCSG yields good empirical performance for the training of multi-layer neural networks. It is worth emphasizing that the mechanism by which SCSG achieves acceleration (variance reduction) is qualitatively different from other speed-up techniques, including momentum \cite{momentum} and adaptive stepsizes \cite{adam}. It will be of interest in future work to explore combinations of these various approaches in the training of deep neural networks.

The rest of paper is organized as follows: In Section \ref{sec:notation} we discuss our notation and assumptions and we state the basic SCSG algorithm. We present the theoretical convergence analysis in Section \ref{sec:ana}. Experimental results are presented in Section \ref{sec:experiments}. All the technical proofs are relegated to the Appendices. Our code is available at \url{https://github.com/Jianbo-Lab/SCSG}.

\section{Notation, Assumptions and Algorithm}\label{sec:notation}

We use $\|\cdot\|$ to denote the Euclidean norm and write $\min\{a, b\}$ as $a\wedge b$ for brevity throughout the paper. The notation $\td{O}$, which hides logarithmic terms, will only be used to maximize readibility in our presentation but will not be used in the formal analysis. 

We define computation cost using the IFO framework of \citet{Agarwal14} which assumes that sampling an index $i$ and accessing the pair $(\nabla f_{i}(x), f_{i}(x))$ incur a unit of cost. For brevity, we write $\nabla f_{\I}(x)$ for $\frac{1}{|\I|}\sum_{i\in \I}\nabla f_{i}(x)$. Note that calculating $\nabla f_{\I}(x)$ incurs $|\I|$ units of computational cost. $x$ is called an $\eps$-accurate solution iff $\E \|\nabla f(x)\|^{2}\le \eps$. The minimum IFO complexity to reach an $\eps$-accurate solution is denoted by $\comp(\eps)$.

Recall that a random variable $N$ has a geometric distribution, $N\sim \mathrm{Geom}(\gamma)$, if $N$ is supported on the non-negative integers \footnote{Here we allow $N$ to be zero to facilitate the analysis.} with
\[P(N = k) = \gamma^{k}(1 - \gamma),\quad \forall k = 0, 1, \ldots\]
An elementary calculation shows that 
\begin{equation}
  \label{eq:Egeom}
  \E_{N\sim \mathrm{Geom}(\gamma)} = \frac{\gamma}{1 - \gamma}.
\end{equation}

To formulate our complexity bounds, we define 
\[f^{*} = \inf_{x}f(x),\quad  \Delta_{f} = f(\td{x}_{0}) - f^{*}.\]
Further we define $\H^{*}$ as an upper bound on the variance of the stochastic gradients:
\begin{equation}\label{eq:Hstar}
\H^{*} = \sup_{x}\frac{1}{n}\sum_{i=1}^{n}\|\nabla f_{i}(x) - \nabla f(x)\|^{2}.
\end{equation}

Assumption \textbf{A}1 on the smoothness of individual functions will be made throughout the paper. 
\begin{enumerate}[\textbf{A}1]
\item $f_{i}$ is differentiable with
\[\|\nabla f_{i}(x) - \nabla f_{i}(y)\|\le L\|x - y\|,\]
for some $L < \infty$ and for all $i\in \{1, \ldots, n\}$.
\end{enumerate}
As a direct consequence of assumption \textbf{A}1, it holds for any $x, y\in \R^{d}$ that
\begin{equation}\label{eq:A1}
-\frac{L}{2}\|x - y\|^{2}\le f_{i}(x) - f_{i}(y) - \la\nabla f_{i}(y), x - y\ra \le \frac{L}{2}\|x - y\|^{2}.
\end{equation}
In this paper, we also consider the Polyak-Lojasiewicz (P-L) condition \cite{polyak63}. It is weaker than strong convexity as well as other popular conditions that appear in the optimization literature; see \cite{karimi16} for an extensive discussion. 
\begin{enumerate}[\textbf{A}2]
\item $f(x)$ satisfies the P-L condition with $\mu > 0$ if
\[\|\nabla f(x)\|^{2}\ge 2\mu(f(x) - f(x^{*}))\]
where $x^{*}$ is the global minimum of $f$.
\end{enumerate}

\subsection{Generic form of SCSG methods}

The algorithm we propose in this paper is similar to that of \cite{PSVRG} except (critically) the number of inner loops is a geometric random variable.  This is an essential component in the analysis of SCSG, and, as we will show below, it is key in allowing us to extend the complexity analysis for SCSG to the non-convex case. Moreover, that algorithm that we present here employs a mini-batch procedure in the inner loop and outputs a random sample instead of an average of the iterates. The pseudo-code is shown in Algorithm \ref{algo:SCSG}.

\begin{algorithm}[htp]
\caption{(Mini-Batch) Stochastically Controlled Stochastic Gradient (SCSG) method for smooth non-convex finite-sum objectives}
\label{algo:SCSG}
\textbf{Inputs: } Number of stages $T$, initial iterate $\td{x}_{0}$, stepsizes $(\eta_{j})_{j=1}^{T}$, batch sizes $(B_{j})_{j=1}^{T}$, mini-batch sizes $(b_{j})_{j=1}^{T}$.

\textbf{Procedure}
\begin{algorithmic}[1]
  \For{$j = 1, 2, \cdots, T$}
  \State Uniformly sample a batch $\I_{j}\subset\{1, \cdots, n\}$ with $|\I_{j}| = B_{j}$;
  \State $g_{j}\gets \nabla f_{\I_{j}}(\td{x}_{j-1})$;
  \State $\xj_{0}\gets \td{x}_{j-1}$;
  \State Generate $N_{j}\sim \mathrm{Geom}\lb B_{j}/(B_{j} + b_{j})\rb$;
  \For{$k = 1, 2, \cdots, N_{j}$}
  \State Randomly pick $\sI_{k-1}\subset [n]$ with $|\sI_{k-1}| = b_{j}$;
  \State $\nuj_{k - 1} \gets \nabla f_{\sI_{k-1}}(\xj_{k-1}) - \nabla f_{\sI_{k-1}}(\xj_{0}) + g_{j}$;
  \State $\xj_{k}\gets \xj_{k-1} - \eta_{j}\nuj_{k - 1}$;
  \EndFor
  \State $\td{x}_{j}\gets \xj_{N_{j}}$;
  \EndFor
\end{algorithmic}
\textbf{Output: } (Smooth case) Sample $\td{x}_{T}^{*}$ from $(\td{x}_{j})_{j=1}^{T}$ with $P(\td{x}_{T}^{*} = \td{x}_{j})\propto \eta_{j}B_{j}/\bj$; (P-L case) $\td{x}_{T}$.
\end{algorithm}

As seen in the pseudo-code, the SCSG method consists of multiple epochs. In the $j$-th epoch, a mini-batch of size $B_{j}$ is drawn uniformly from the data and a sequence of mini-batch SVRG-type updates are implemented, with the total number of updates being randomly generated from a geometric distribution, with mean equal to the batch size. Finally it outputs a random sample from $\{\td{x}_{j}\}_{j=1}^{T}$. This is the standard way, proposed by \cite{Nemirovsky09}, as opposed to computing $\argmin_{j\le T} \|\nabla f(\td{x}_{j})\|$ which requires additional overhead. By \eqref{eq:Egeom}, the average total cost is 
\begin{equation}\label{eq:avgcost}
\sum_{j=1}^{T}(B_{j} + b_{j}\cdot\E N_{j}) = \sum_{i=1}^{T}(B_{j} + b_{j}\cdot \frac{B_{j}}{b_{j}}) = 2\sum_{j=1}^{T}B_{j}.
\end{equation}
Define $T(\eps)$ as the minimum number of epochs such that all outputs afterwards are $\eps$-accurate solutions, i.e.
\[T(\eps) = \min\{T: \E\|\nabla f(\td{x}_{T'}^{*})\|^{2}\le \eps \mbox{ for all }T' \ge T\}.\]
Recall the definition of $\comp(\eps)$ at the beginning of this section, the average IFO complexity to reach an $\eps$-accurate solution is 
\[\E \comp(\eps) \le 2\sum_{j=1}^{T(\eps)}B_{j}.\]

\subsection{Parameter settings}
The generic form (Algorithm \ref{algo:SCSG}) allows for flexibility in both stepsize, $\etaj$, and batch/mini-batch size, $(\Bj, \bj)$. In order to minimize the amount of tuning needed in practice, we provide several default settings which have theoretical support. The settings and the corresponding complexity results are summarized in Table \ref{tab:setup}. Note that all settings fix $\bj = 1$ since this yields the best rate as will be shown in Section \ref{sec:ana}. However, in practice a reasonably large mini-batch size $\bj$ might be favorable due to the acceleration that could be achieved by vectorization; see Section \ref{sec:experiments} for more discussions on this point.

\begin{table}
  \centering
  \caption{Parameter settings analyzed in this paper.}\label{tab:setup}
  \begin{tabular}{c|c|c|c|c|c}
    \toprule
    & $\etaj$ & $\Bj$ & $\bj$ & Type of Objectives & $\E \comp(\eps)$\\
    \midrule
Version 1 & $\frac{1}{6LB^{2/3}}$ & $O\lb\frac{1}{\eps}\wedge n\rb$ & $1$ & Smooth & $O\lb\frac{1}{\eps^{5/3}}\wedge \frac{n^{2/3}}{\eps}\rb$\\
Version 2 & $\frac{1}{6L\Bj^{2/3}}$ & $j^{\frac{3}{2}}\wedge n$ & $1$ & Smooth & $\td{O}\lb\frac{1}{\eps^{5/3}}\wedge \frac{n^{2/3}}{\eps}\rb$\\
Version 3 & $\frac{1}{6L\Bj^{2/3}}$ & $O\lb\frac{1}{\mu\eps}\wedge n\rb$ & $1$ & Polyak-Lojasiewicz & $\td{O}\lb(\frac{1}{\mu\eps}\wedge n) + \frac{1}{\mu}(\frac{1}{\mu\eps}\wedge n)^{2/3}\rb$\\ 
    \bottomrule
  \end{tabular}
\end{table}

\section{Convergence Analysis}\label{sec:ana}

\subsection{One-epoch analysis}
First we present the analysis for a single epoch. Given $j$, we define
\begin{equation}
  \label{eq:ej}
  \ej = \nabla f_{\I_{j}}(\td{x}_{j-1}) - \nabla f(\td{x}_{j-1}).
\end{equation}

As shown in \cite{PSVRG}, the gradient update $\nuj_{k}$ is a biased estimate of the gradient $\nabla f(\xj_{k})$ conditioning on the current random index $i_{k}$. Specifically, within the $j$-th epoch,
\[\E_{\sI_{k}}\nuj_{k} = \nabla f(\xj_{k}) + \nabla f_{\I_{j}}(\xj_{0}) - \nabla f(\xj_{0}) = \nabla f(\xj_{k}) + \ej.\]
This reveals the basic qualitative difference between SVRG and SCSG. Most of the novelty in our analysis lies in dealing with the extra term $\ej$. Unlike \cite{PSVRG}, we do not assume $\|\xj_{k} - x^{*}\|$ to be bounded since this is invalid in unconstrained problems, even in convex cases. 

By careful analysis of primal and dual gaps~\cite{zhu14}, we find that the stepsize $\etaj$ should scale as $(\Bj / \bj)^{-\frac{2}{3}}$. Then same phenomenon has also been observed in \citet{reddi16svrg, reddi16saga, Zhu16} when $\bj = 1$ and $\Bj = n$. 

\begin{theorem}\label{thm:one_epoch}
Let $\etaj L = \gamma(\Bj / \bj) ^{-\frac{2}{3}}$. Suppose $\gamma \le \frac{1}{3}$ and $\Bj\ge 8\bj$ for all $j$, then under Assumption \textbf{A}1, 
\begin{equation}\label{eq:one_epoch}
\E \|\nabla f(\td{x}_{j})\|^{2}\le \frac{5L}{\gamma}\cdot \lb\frac{\bj}{\Bj}\rb^{\frac{1}{3}}\E (f(\td{x}_{j-1}) - f(\td{x}_{j})) + \frac{6I(\Bj < n)}{\Bj}\cdot\H^{*}.
\end{equation}
\end{theorem}

The proof is presented in Appendix \ref{app:one_epoch}. It is not surprising that a large mini-batch size will increase the theoretical complexity as in the analysis of mini-batch SGD. For this reason we restrict most of our subsequent analysis to $\bj \equiv 1$.

\subsection{Convergence analysis for smooth non-convex objectives}
When only assuming smoothness, the output $\td{x}_{T}^{*}$ is a random element from $(\td{x}_{j})_{j=1}^{T}$. Telescoping \eqref{eq:one_epoch} over all epochs, we easily obtain the following result.

\begin{theorem}\label{thm:smooth}
Under the specifications of Theorem \ref{thm:one_epoch} and Assumption \textbf{A}1, 
\[\E \|\nabla f(\td{x}_{T}^{*})\|^{2}\le \frac{\frac{5L}{\gamma}\Delta_{f} + 6\lb\sum_{j=1}^{T}\bj^{-\frac{1}{3}}\Bj^{-\frac{2}{3}}I(\Bj < n)\rb\H^{*}}{\sum_{j=1}^{T}\bj^{-\frac{1}{3}}\Bj^{\frac{1}{3}}}.\]
\end{theorem}

This theorem covers many existing results. When $\Bj = n$ and $\bj = 1$, Theorem \ref{thm:smooth} implies that $\E \|\nabla f(\td{x}_{T}^{*})\|^{2} = O\lb\frac{L\Delta_{f}}{Tn^{1/3}}\rb$ and hence $T(\eps) = O(1 + \frac{L\Delta_{f}}{\eps n^{1/3}})$. This yields the same complexity bound $\E \comp(\eps) = O(n + \frac{n^{2/3}L\Delta_{f}}{\eps})$ as SVRG \cite{reddi16svrg}. On the other hand, when $\bj = \Bj\equiv B$ for some $B < n$, Theorem \ref{thm:smooth} implies that $\E \|\nabla f(\td{x}_{T}^{*})\|^{2} = O\lb\frac{L\Delta_{f}}{T} + \frac{\H^{*}}{B}\rb$. The second term can be made $O(\eps)$ by setting $B = O\lb\frac{\H^{*}}{\eps}\rb$. Under this setting $T(\eps) = O\lb\frac{L\Delta_{f}}{\eps}\rb$ and $\E \comp(\eps) = O\lb\frac{L\Delta_{f}\H^{*}}{\eps^{2}}\rb$. This is the same rate as in \cite{reddi16svrg} for SGD. 

However, both of the above settings are suboptimal since they either set the batch sizes $\Bj$ too large or set the mini-batch sizes $\bj$ too large. By Theorem \ref{thm:smooth}, SCSG can be regarded as an interpolation between SGD and SVRG. By leveraging these two parameters, SCSG is able to outperform both methods. 

We start from considering a constant batch/mini-batch size $\Bj\equiv B, \bj \equiv 1$. Similar to SGD and SCSG, $B$ should be at least $O (\frac{\H^{*}}{\eps})$. In applications like the training of neural networks, the required accuracy is moderate and hence a small batch size suffices. This is particularly important since the gradient can be computed without communication overhead, which is the bottleneck of SVRG-type algorithms. As shown in Corollary \ref{cor:smooth_constant} below, the complexity of SCSG beats both SGD and SVRG.

\begin{corollary}\label{cor:smooth_constant}(Constant batch sizes)
  Set 
\[\bj \equiv 1, \quad \Bj \equiv B = \min\left\{\frac{12\H^{*}}{\eps}, n\right\}, \quad \etaj \equiv\eta = \frac{1}{6L B^{\frac{2}{3}}}.\]
Then it holds that 
\[\E \comp(\eps) = O\lb \lb\frac{\H^{*}}{\eps}\wedge n\rb + \frac{L\Delta_{f}}{\eps}\cdot \lb\frac{\H^{*}}{\eps}\wedge n\rb^{\frac{2}{3}}\rb.\]
\end{corollary}

Assume that $L\Delta_{f}, \H^{*} = O(1)$, the above bound can be simplified to
\[\E \comp(\eps) = O\lb \lb\frac{1}{\eps}\wedge n\rb + \frac{1}{\eps}\cdot \lb\frac{1}{\eps}\wedge n\rb^{\frac{2}{3}}\rb = O\lb\frac{1}{\eps^{\frac{5}{3}}}\wedge \frac{n^{\frac{2}{3}}}{\eps}\rb.\]
When the target accuracy is high, one might consider a sequence of increasing batch sizes. Heuristically, a large batch is wasteful at the early stages when the iterates are inaccurate. Fixing the batch size to be $n$ as in SVRG is obviously suboptimal. Via an involved analysis, we find that $\Bj \sim j^{\frac{3}{2}}$ gives the best complexity among the class of SCSG algorithms. 

\begin{corollary}\label{cor:smooth_vary} (Time-varying batch sizes)
Set 
\[\bj\equiv 1, \quad\Bj = \min\left\{\lceil j^{\frac{3}{2}}\rceil, n\right\}, \quad \etaj  = \frac{1}{6L\Bj^{\frac{2}{3}}}.\]
Then it holds that
\begin{align}
\E \comp(\eps) = O\lb \min\left\{\frac{1}{\eps^{\frac{5}{3}}}\left[(L\Delta_{f})^{\frac{5}{3}} + (\H^{*})^{\frac{5}{3}}\log^{5} \lb\frac{\H^{*}}{\eps}\rb\right], n^{\frac{5}{3}}\right\} + \frac{n^{\frac{2}{3}}}{\eps}\cdot (L\Delta_{f} + \H^{*}\log n)\rb.\label{eq:comp_smooth_best}
\end{align}
\end{corollary}

The proofs of both Corollary \ref{cor:smooth_constant} and Corollary \ref{cor:smooth_vary} are presented in Appendix \ref{app:ana_smooth}. To simplify the bound \eqref{eq:comp_smooth_best}, we assume that $\Delta_{f}, \H^{*}= O(1)$ in order to highlight the dependence on $\eps$ and $n$. Then \eqref{eq:comp_smooth_best} can be simplified to
\[\E \comp(\eps) = O\lb\frac{1}{\eps^{\frac{5}{3}}}\log^{5}\lb\frac{1}{\eps}\rb \wedge n^{\frac{5}{3}} + \frac{n^{\frac{2}{3}}\log n}{\eps}\rb = \td{O}\lb\frac{1}{\eps^{\frac{5}{3}}}\wedge n^{\frac{5}{3}} + \frac{n^{\frac{2}{3}}}{\eps}\rb = \td{O}\lb \frac{1}{\eps^{\frac{5}{3}}}\wedge \frac{n^{\frac{2}{3}}}{\eps}\rb.\]
The log-factor $\log^{5}\lb\frac{1}{\eps}\rb$ is purely an artifact of our proof. It can be reduced to $\log^{\frac{3}{2} + \mu}\lb\frac{1}{\eps}\rb$ for any $\mu > 0$ by setting $\Bj \sim j^{\frac{3}{2}}(\log j)^{\frac{3}{2}+ \mu}$; see remark \ref{rem:smooth_log} in Appendix \ref{app:ana_smooth}. 

\subsection{Convergence analysis for P-L objectives}
When the component $f_{i}(x)$ satisfies the P-L condition, it is known that the global minimum can be found efficiently by SGD \cite{karimi16} and SVRG-type algorithms \cite{reddi16svrg, Zhu16}. Similarly, SCSG can also achieve this. As in the last subsection, we start from a generic result to bound $\E (f(\td{x}_{T}) - f^{*})$ and then consider specific settings of the parameters as well as their complexity bounds.
\begin{theorem}\label{thm:PL}
  Let $\lambda_{j} = \frac{5L\bj^{\frac{1}{3}}}{\mu \gamma\Bj^{\frac{1}{3}} + 5L\bj^{\frac{1}{3}}}$. Then under the same settings of Theorem \ref{thm:smooth}, 
  \[\E (f(\td{x}_{T}) - f^{*})\le \lambda_{T}\lambda_{T-1}\ldots \lambda_{1}\cdot\Delta_{f} + 6\gamma\H^{*}\cdot \sum_{j=1}^{T}\frac{\lambda_{T}\lambda_{T-1}\ldots \lambda_{j+1}\cdot I(\Bj < n)}{\mu \gamma\Bj + 5L\bj^{\frac{1}{3}}\Bj^{\frac{2}{3}}}.\]
\end{theorem}
The proofs and additional discussion are presented in Appendix \ref{app:ana_PL}. Again, Theorem \ref{thm:PL} covers existing complexity bounds for both SGD and SVRG. In fact, when $\Bj = \bj\equiv B$ as in SGD, via some calculation, we obtain that 
\[\E (f(\td{x}_{T}) - f^{*}) = O\lb \lb\frac{L}{\mu + L}\rb^{T}\cdot \Delta_{f} + \frac{\H^{*}}{\mu B}\rb.\]
The second term can be made $O(\eps)$ by setting $B = O(\frac{\H^{*}}{\mu\eps})$, in which case $T(\eps) = O(\frac{L}{\mu}\log \frac{\Delta_{f}}{\eps})$. As a result, the average cost to reach an $\eps$-accurate solution is $\E \comp(\eps) = O(\frac{L\H^{*}}{\mu^{2}\eps})$, which is the same as \citet{karimi16}. On the other hand, when $\Bj \equiv n$ and $\bj \equiv 1$ as in SVRG, Theorem \ref{thm:PL} implies that 
\[\E (f(\td{x}_{T}) - f^{*}) = O\lb \lb\frac{L}{\mu n^{\frac{1}{3}} + L}\rb^{T}\cdot \Delta_{f}\rb.\]
This entails that $T(\eps) = O\lb (1 + \frac{1}{\mu n^{1/3}})\log\frac{1}{\eps}\rb$ and hence $\E \comp(\eps) = O\lb (n + \frac{n^{2/3}}{\mu})\log\frac{1}{\eps}\rb$, which is the same as \citet{reddi16svrg}. 

By leveraging the batch and mini-batch sizes, we obtain a counterpart of Corollary \ref{cor:smooth_constant} as below.
\begin{corollary}\label{cor:PL_constant}
  Set 
\[\bj \equiv 1, \quad \Bj \equiv B = \min\left\{\frac{12\H^{*}}{\mu\eps}, n\right\}, \quad \etaj \equiv\eta = \frac{1}{6LB^{\frac{2}{3}}}\]
Then it holds that 
\[\E \comp(\eps) = O\lb \left\{\lb\frac{\H^{*}}{\mu \eps}\wedge n\rb + \frac{1}{\mu}\lb\frac{\H^{*}}{\mu \eps}\wedge n\rb^{\frac{2}{3}}\right\}\log \frac{\Delta_{f}}{\eps}\rb.\]  
\end{corollary}

Recall the results from Table \ref{tab:summary},  SCSG is $O\lb\frac{1}{\mu} + \frac{1}{(\mu\eps)^{1/3}}\rb$ faster than SGD and is never worse than SVRG. When both $\mu$ and $\eps$ are moderate, the acceleration of SCSG over SVRG is significant. Unlike the smooth case, we do not find any possible choice of setting that can achieve a better rate than Corollary \ref{cor:PL_constant}.

\section{Experiments}\label{sec:experiments}
We evaluate SCSG and mini-batch SGD on the MNIST dataset with (1) a three-layer fully-connected neural network with $512$ neurons in each layer (FCN for short) and (2) a standard convolutional neural network LeNet \cite{lecun1998gradient} (CNN for short), which has two convolutional layers with $32$ and $64$ filters of size $5\times 5$ respectively, followed by two fully-connected layers with output size $1024$ and $10$. Max pooling is applied after each convolutional layer. The MNIST dataset of handwritten digits has $50,000$ training examples and $10,000$ test examples. The digits have been size-normalized and centered in a fixed-size image. Each image is $28$ pixels by $28$ pixels. All experiments were carried out on an Amazon p2.xlarge node with a NVIDIA GK210 GPU with algorithms implemented in TensorFlow 1.0. 

Due to the memory issues, sampling a chunk of data is costly. We avoid this by modifying the inner loop: instead of sampling mini-batches from the whole dataset, we split the batch $\I_{j}$ into $\Bj / \bj$ mini-batches and run SVRG-type updates sequentially on each. Despite the theoretical advantage of setting $\bj = 1$, we consider practical settings $\bj > 1$ to take advantage of the acceleration obtained by vectorization. We initialized parameters by TensorFlow's default Xavier uniform initializer. In all experiments below, we show the results corresponding to the best-tuned stepsizes. 


We consider three algorithms: (1) SGD with a fixed batch size $B\in  \{512, 1024\}$; (2) SCSG with a fixed batch size $B\in \{512, 1024\}$ and a fixed mini-batch size $b = 32$; (3) SCSG with time-varying batch sizes $\Bj = \lceil j^{3/2}\wedge n\rceil$ and $\bj = \lceil \Bj / 32\rceil$. To be clear, given $T$ epochs, the IFO complexity of the three algorithms are $TB$, $2TB$ and $2\sum_{j=1}^{T}\Bj$, respectively. We run each algorithm with 20 passes of data. It is worth mentioning that the largest batch size in Algorithm 3 is $\lceil 275^{1.5}\rceil = 4561$, which is relatively small compared to the sample size $50000$.


We plot in Figure \ref{fig:res1} the training and the validation loss against the IFO complexity---i.e., the number of passes of data---for fair comparison. In all cases, both versions of SCSG outperform SGD, especially in terms of training loss. SCSG with time-varying batch sizes always has the best performance and it is more stable than SCSG with a fixed batch size. For the latter, the acceleration is more significant after increasing the batch size to $1024$. Both versions of SCSG provide strong evidence that variance reduction can be achieved efficiently without evaluating the full gradient. 


\begin{figure}[ht]
    \centering
        \includegraphics[width=0.24\textwidth]{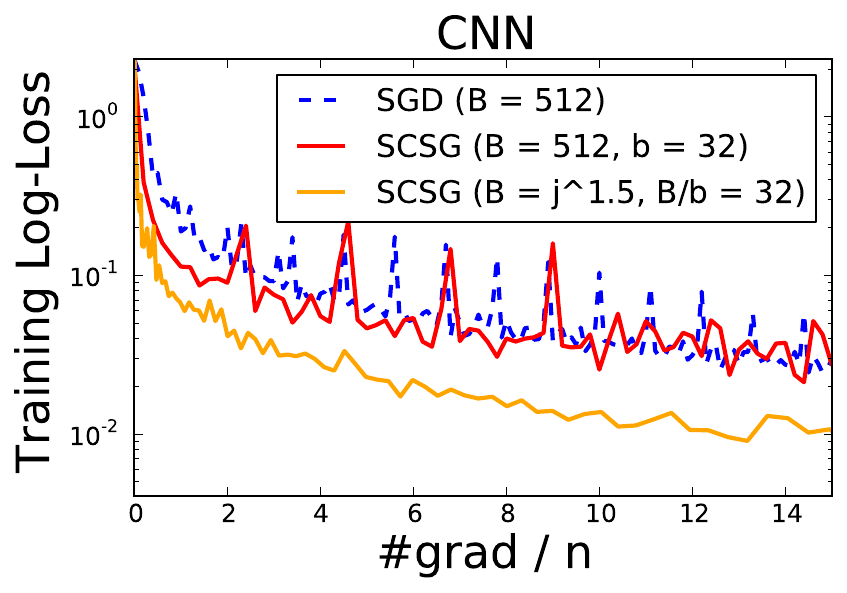}
        \includegraphics[width=0.24\textwidth]{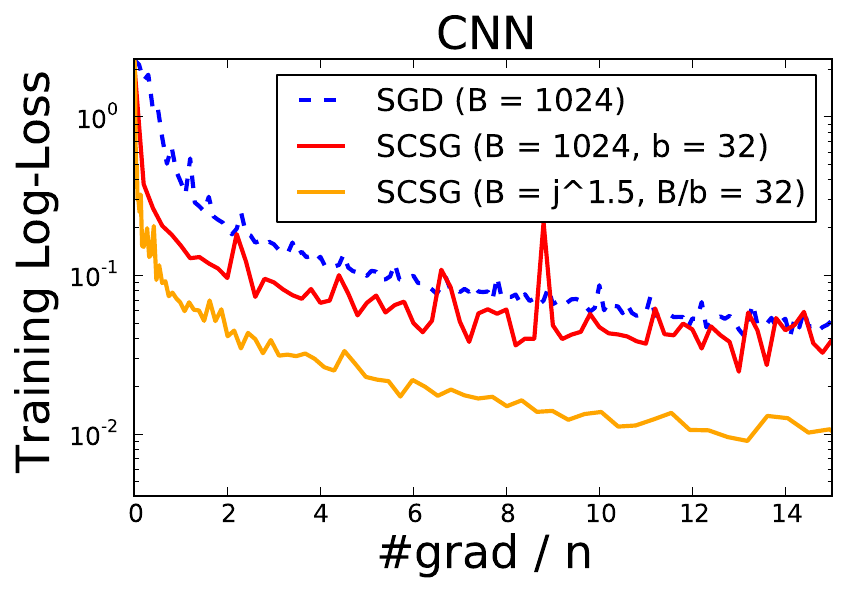}
        \includegraphics[width=0.24\textwidth]{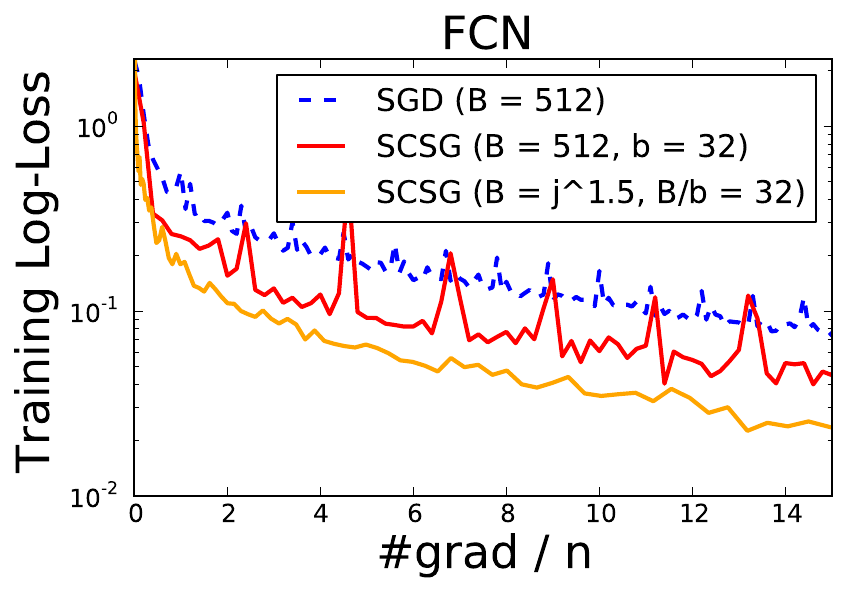}
        \includegraphics[width=0.24\textwidth]{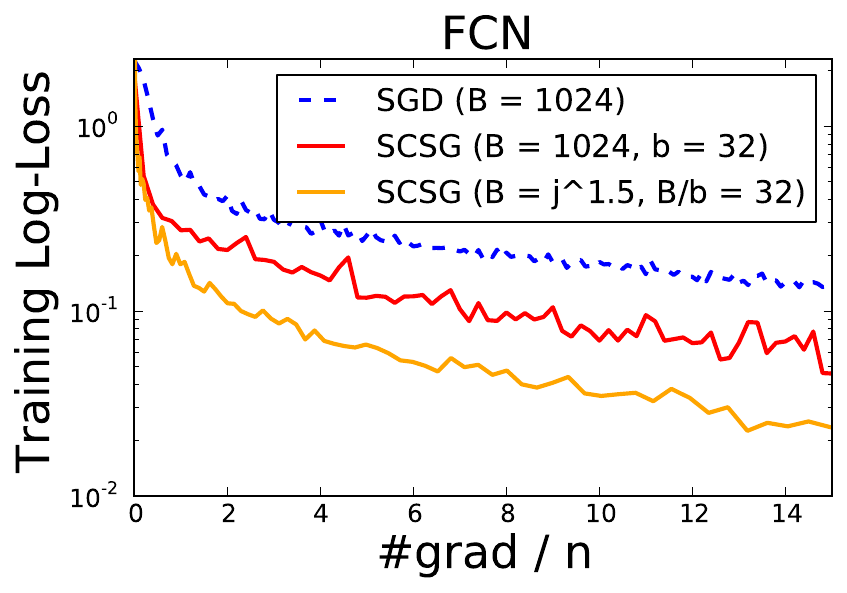}
        \includegraphics[width=0.24\textwidth]{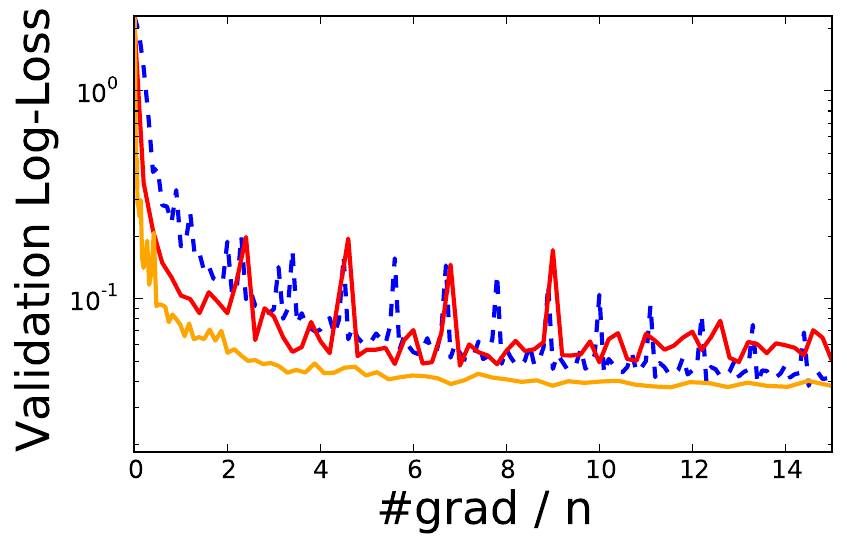}
        \includegraphics[width=0.24\textwidth]{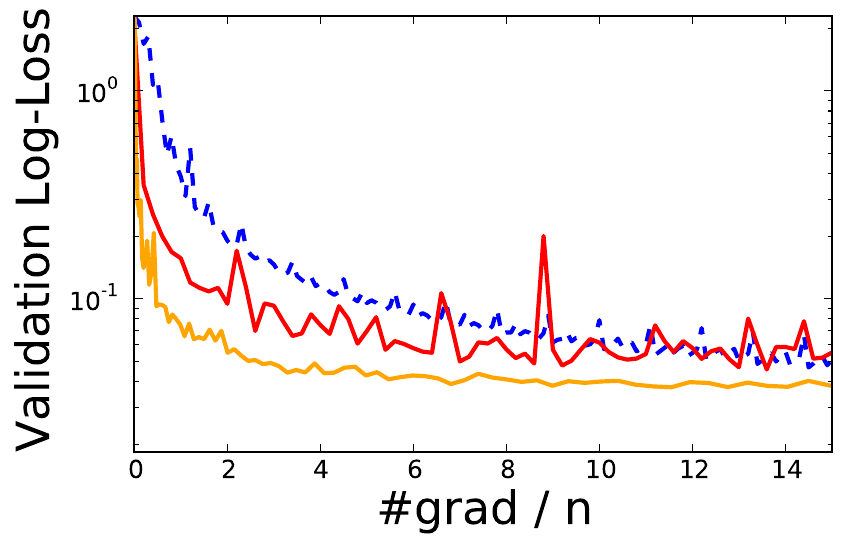}
        \includegraphics[width=0.24\textwidth]{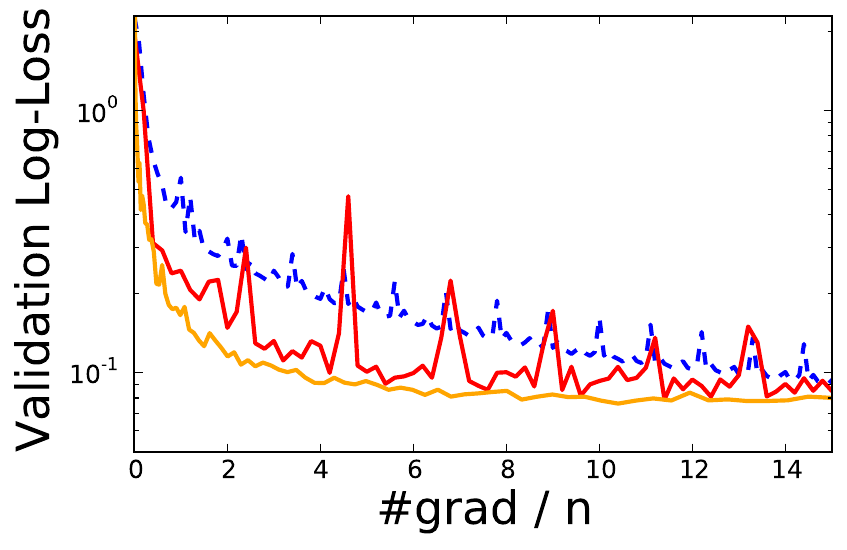}
        \includegraphics[width=0.24\textwidth]{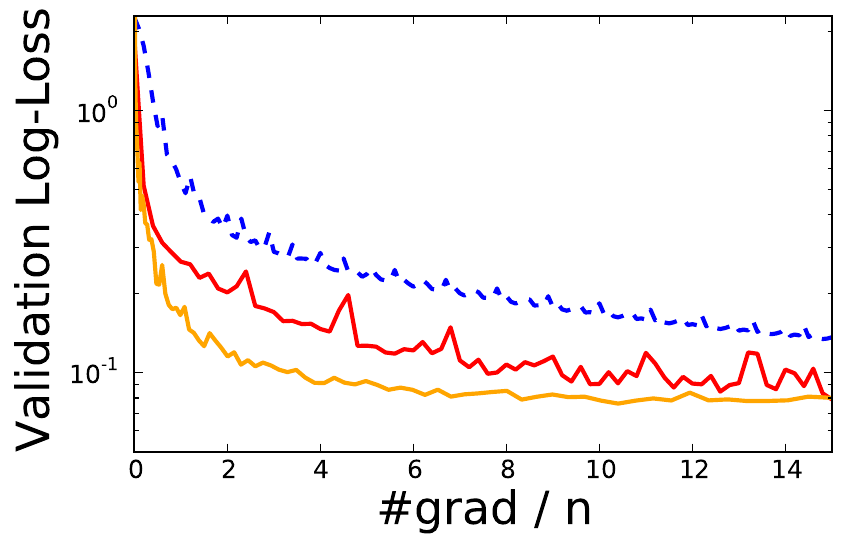}
        \caption{Comparison between two versions of SCSG and mini-batch SGD of training loss (top row) and validation loss (bottom row) against the number of IFO calls. The loss is plotted on a log-scale. Each column represents an experiment with the setup printed on the top.}
    \label{fig:res1}
\end{figure}

Given $2B$ IFO calls, SGD implements updates on two fresh batches while SCSG replaces the second batch by a sequence of variance reduced updates. Thus, Figure \ref{fig:res1} shows that the gain due to variance reduction is significant when the batch size is fixed. To further explore this, we compare SCSG with time-varying batch sizes to SGD with the same sequence of batch sizes. The results corresponding to the best-tuned constant stepsizes are plotted in Figure \ref{fig:res2}. It is clear that the benefit from variance reduction is more significant when using time-varying batch sizes.

We also compare the performance of SGD with that of SCSG with time-varying batch sizes against wall clock time, when both algorithms are implemented in TensorFlow and run on a Amazon p2.xlarge node with a NVIDIA GK210 GPU. Due to the cost of computing variance reduction terms in SCSG, each update of SCSG is slower per iteration compared to SGD. However, SCSG makes faster progress in terms of both training loss and validation loss compared to SCD in wall clock time. The results are shown in Figure \ref{fig:clock_time}. 

\begin{figure}[ht]
    \centering
        \includegraphics[width=0.45\textwidth]{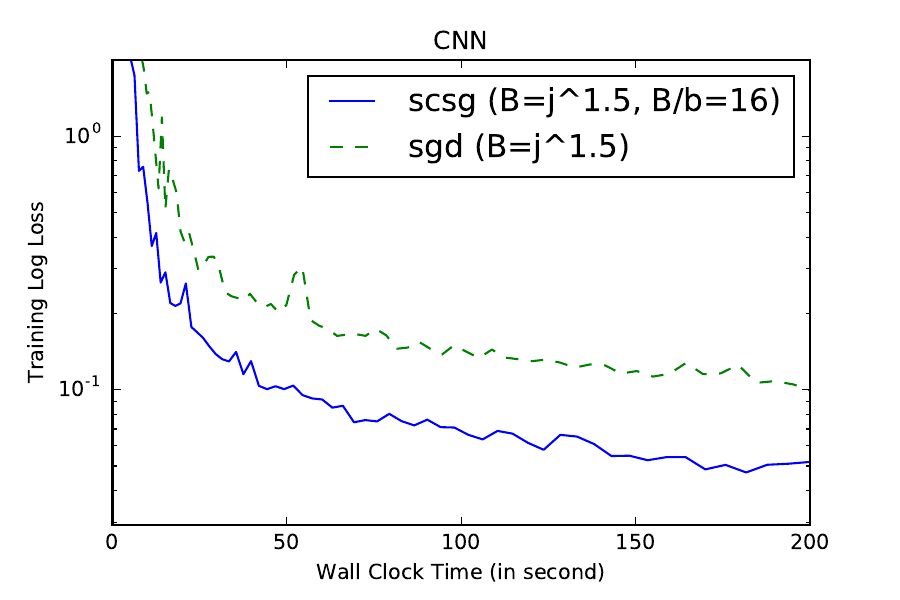}
        \includegraphics[width=0.45\textwidth]{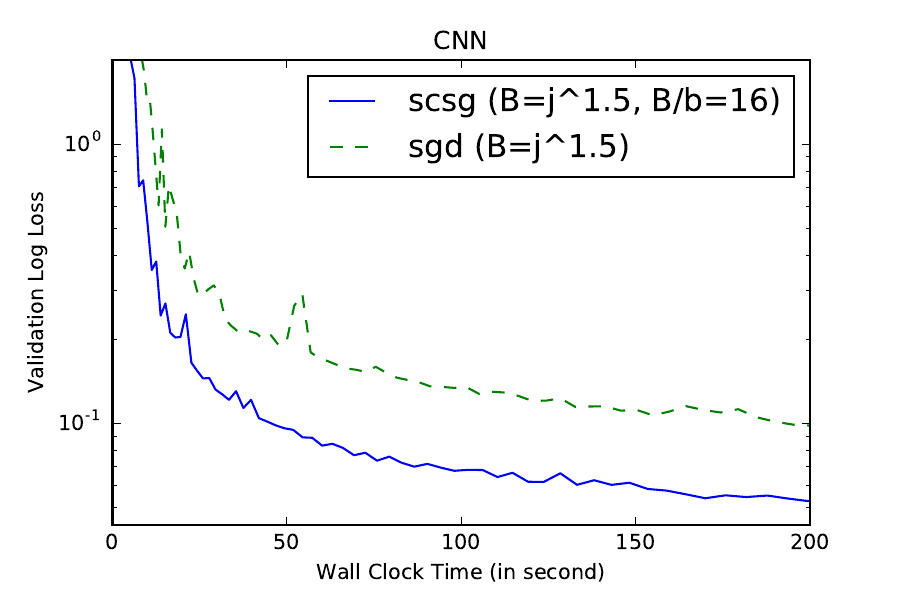} 
        \caption{Comparison between SCSG and mini-batch SGD of training loss and validation loss with a CNN loss, against wall clock time. The loss is plotted on a log-scale.}
    \label{fig:clock_time}
\end{figure}

Finally, we examine the effect of $\Bj / \bj$, namely the number of mini-batches within an iteration, since it affects the efficiency in practice where the computation time is not proportional to the batch size. Figure \ref{fig:res3} shows the results for SCSG  with $\Bj= \lceil j^{3/2}\wedge n\rceil$ and $\lceil \Bj / \bj\rceil\in \{2, 5, 10, 16, 32\}$. In general, larger $\Bj / \bj$ yields better performance. It would be interesting to explore the tradeoff between computation efficiency and this ratio on different platforms.

\begin{figure}[ht]
    \centering
     \begin{subfigure}[b]{0.49\textwidth}
        \includegraphics[width=0.49\textwidth]{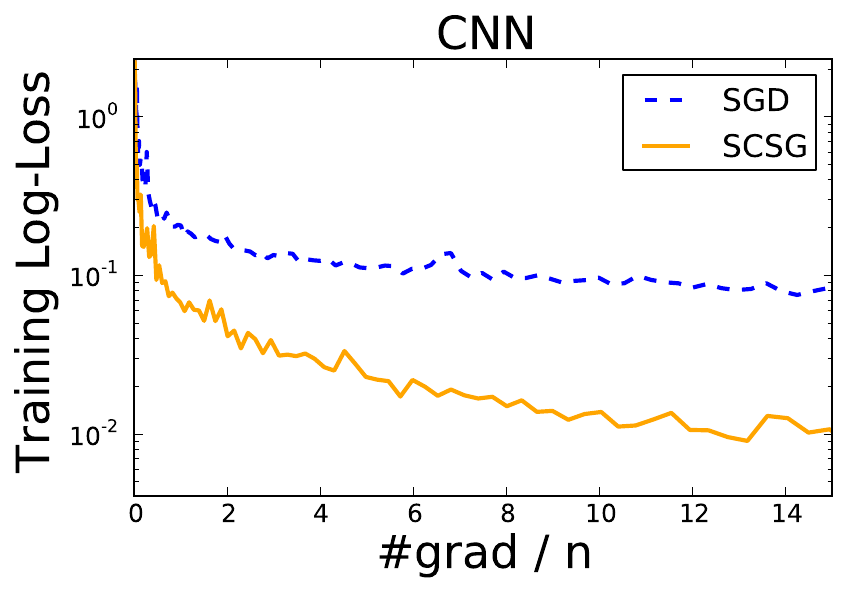}
        \includegraphics[width=0.49\textwidth]{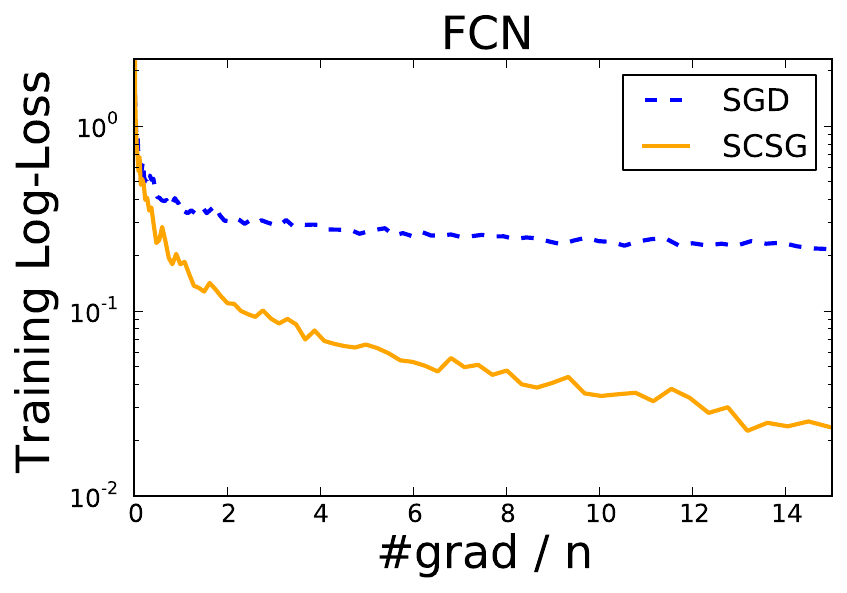}
        \caption{SCSG and SGD with increasing batch sizes}\label{fig:res2}
      \end{subfigure}\hfill
      \begin{subfigure}[b]{0.49\textwidth}
        \includegraphics[width=0.49\textwidth]{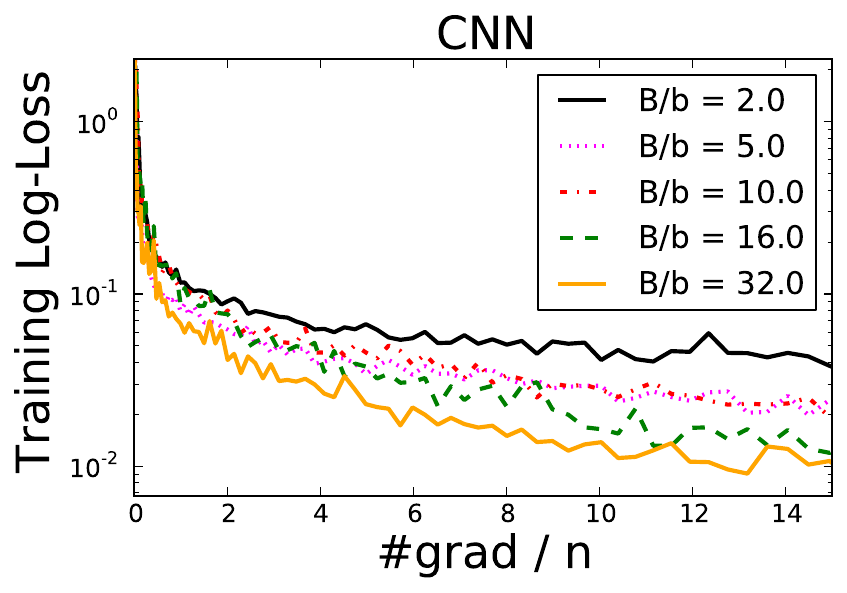}
        \includegraphics[width=0.49\textwidth]{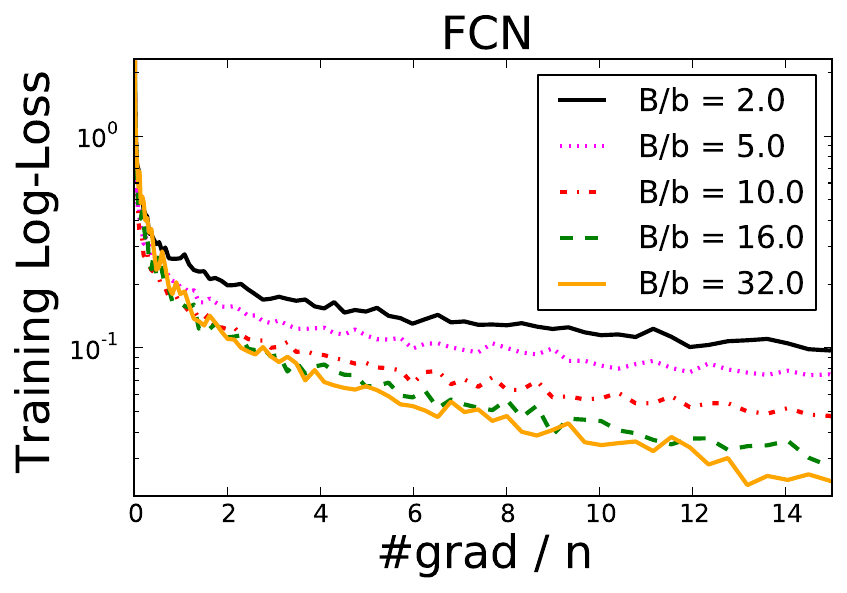}
        \caption{SCSG with different $\Bj / \bj$}\label{fig:res3}
      \end{subfigure}
\end{figure}

\section{Discussion}
We have presented the SCSG method for smooth, non-convex, finite-sum optimization problems. SCSG is the first algorithm that achieves a uniformly better rate than SGD and is never worse than SVRG-type algorithms. When the target accuracy is low, SCSG significantly outperforms the SVRG-type algorithms. Unlike various other variants of SVRG, SCSG is clean in terms of both implementation and analysis. Empirically, SCSG outperforms SGD in the training of multi-layer neural networks.

Although we only consider the finite-sum objective in this paper, it is straightforward to extend SCSG to the general stochastic optimization problems where the objective can be written as $\E_{\xi\sim F} f(x; \xi)$: at the beginning of $j$-th epoch a batch of i.i.d. sample $(\xi_{1}, \ldots, \xi_{B_{j}})$ is drawn from the distribution $F$ and 
\[g_{j} = \frac{1}{B_{j}}\sum_{i=1}^{B_{j}}\nabla f(\td{x}_{j-1}; \xi_{i})\quad \mbox{(see line 3 of Algorithm \ref{algo:SCSG})};\] 
at the $k$-th step, a fresh sample $(\td{\xi}_{1}^{(k)}, \ldots, \td{\xi}_{b_{j}}^{(k)})$ is drawn from the distribution $F$ and 
\[\nu_{k-1}^{(j)} = \frac{1}{b_{j}}\sum_{i=1}^{b_{j}}\nabla f(x_{k-1}^{(j)}; \td{\xi}_{i}^{(k)}) - \frac{1}{b_{j}}\sum_{i=1}^{b_{j}}\nabla f(x_{0}^{(j)}; \td{\xi}_{i}^{(k)}) + g_{j}\quad \mbox{(see line 8 of Algorithm \ref{algo:SCSG})}.\]
Our proof directly carries over to this case, by simply suppressing the term $I(B_{j} < n)$, and yields the bound $\td{O}(\eps^{-5/3})$ for smooth non-convex objectives and the bound $\td{O}(\mu^{-1}\eps^{-1}\wedge \mu^{-5/3}\eps^{-2/3})$ for P-L objectives. These bounds are simply obtained by setting $n = \infty$ in our convergence analysis.

Compared to momentum-based methods \cite{momentum} and methods with adaptive stepsizes \cite{Adagrad, adam}, the mechanism whereby SCSG achieves acceleration is qualitatively different: while momentum aims at balancing primal and dual gaps \cite{zhu14}, adaptive stepsizes aim at balancing the scale of each coordinate, and variance reduction aims at removing the noise. We believe that an algorithm that combines these three techniques is worthy of further study, especially in the training of deep neural networks where the target accuracy is moderate. 

\section*{Acknowledgments}
The authors thank Zeyuan Allen-Zhu, Chi Jin, Nilesh Tripuraneni, Yi Xu, Tianbao Yang, Shenyi Zhao and anonymous reviewers for helpful discussions.

\bibliography{SCSG}
\bibliographystyle{plain}

\newpage
\appendix

\section{Technical Lemmas}\label{app:lemmas}
In this section we present several technical lemmas that facilitate the proofs of our main results. 

We start with a lemma on the variance of the sample mean (without replacement). 
\begin{lemma}\label{lem:var_sampling}
Let $x_{1}, \ldots, x_{M}\in \R^{d}$ be an arbitrary population of $N$ vectors with 
\[\sum_{j=1}^{M}x_{j} = 0.\]
Further let $\mathcal{J}$ be a uniform random subset of $\{1, \ldots, M\}$ with size $m$. Then 
\[\E \left\|\frac{1}{m}\sum_{j\in \mathcal{J}}x_{j}\right\|^{2} = \frac{M - m}{(M - 1)m}\cdot \frac{1}{M}\sum_{j=1}^{M}\|x_{j}\|^{2}\le \frac{I(m < M)}{m}\cdot \frac{1}{M}\sum_{j=1}^{M}\|x_{j}\|^{2}.\]
\end{lemma}
\begin{proof}
Let $W_{j} = I(j\in \mathcal{J})$, then it is easy to see that
\begin{equation}
  \label{eq:cov_wj}
  \E W_{j}^{2} = \E W_{j} = \frac{m}{M}, \quad \E W_{j}W_{j'} = \frac{m(m - 1)}{M(M - 1)}.
\end{equation}
Then the sample mean can be rewritten as
\[\frac{1}{m}\sum_{j\in \mathcal{J}}x_{j} = \frac{1}{m}\sum_{i=1}^{n}W_{j}x_{j}.\]
This implies that 
\begin{align*}
\E \bigg\|\frac{1}{m}\sum_{j\in \mathcal{J}}x_{j}\bigg\|^{2} &= \frac{1}{m^{2}}\lb\sum_{j=1}^{M}\E W_{j}^{2}\|x_{j}\|^{2} + \sum_{j\not= j'}\E W_{j}W_{j'} \la x_{j}, x_{j'}\ra\rb\\
& = \frac{1}{m^{2}}\lb\frac{m}{M}\sum_{j=1}^{M}\|x_{j}\|^{2} + \frac{m(m - 1)}{M(M - 1)}\sum_{j\not= j'} \la x_{j}, x_{j'}\ra\rb\\
& = \frac{1}{m^{2}}\lb\lb\frac{m}{M} - \frac{m(m - 1)}{M(M - 1)}\rb\sum_{j=1}^{M}\|x_{j}\|^{2} + \frac{m(m - 1)}{M(M - 1)}\left\|\sum_{j=1}^{M}x_{j}\right\|^{2}\rb\\
& = \frac{1}{m^{2}}\lb\frac{m}{M} - \frac{m(m - 1)}{M(M - 1)}\rb\sum_{j=1}^{M}\|x_{j}\|^{2}\\
& = \frac{M - m}{(M - 1)m}\cdot\frac{1}{M}\sum_{j=1}^{M}\|x_{j}\|^{2}.
\end{align*}
\end{proof}

Since the geometric random variable $N_{j}$ plays an important role in our analysis, we recall its key properties below. 
\begin{lemma}\label{lem:geom}
Let $N\sim \mathrm{Geom}(\gamma)$ for $\gamma > 0$. Then for any sequence $D_{0}, D_{1}, \ldots$ with $\E |D_{N}| < \infty$
\[\E (D_{N} - D_{N + 1}) = \lb\frac{1}{\gamma} - 1\rb\lb D_{0} - \E D_{N}\rb.\]
\end{lemma}
\begin{proof}
By definition, 
  \begin{align*}
   &\E (D_{N} - D_{N + 1})  = \sum_{n\ge 0}(D_{n} - D_{n+1})\cdot \gamma^{n}(1 - \gamma)\\ 
= &(1 - \gamma)\lb D_{0} - \sum_{n\ge 1}D_{n}(\gamma^{n-1} - \gamma^{n})\rb = (1 - \gamma)\lb \frac{1}{\gamma}D_{0} - \sum_{n\ge 0}D_{n}(\gamma^{n-1} - \gamma^{n})\rb\\ 
= & (1 - \gamma)\lb \frac{1}{\gamma}D_{0} - \frac{1}{\gamma}\sum_{n\ge 0}D_{n}\gamma^{n}(1 - \gamma)\rb = \lb\frac{1}{\gamma} - 1\rb (D_{0} - \E D_{N}),
  \end{align*}
where the last equality is implied by the condition that $\E |D_{N}| < \infty$.
\end{proof}

\begin{lemma}\label{lem:log}
  For any $\eta > 1$ and $z > 0$, define $g_{\eta}(x)$ and $x(z)$ as 
\[g_{\eta}(x) = \frac{1 + \log x}{x^{\eta}}, \quad x(z) = z^{-\frac{1}{\eta}}\cdot \lb \frac{2}{\eta}\log \frac{1}{z}\vee 2\rb^{\frac{1}{\eta - 1}}.\]
Then 
\[g_{\eta}(x) \le z \quad \forall x \ge x(z).\]
\end{lemma}
\begin{proof}
For any $x \ge x(z)$, denote $\alpha = x/z^{-\frac{1}{\eta}}$. Then 
\[\alpha \ge \lb \frac{2}{\eta}\log \frac{1}{z}\vee 2\rb^{\frac{1}{\eta - 1}}\ge 1\]
and 
\[g_{\eta}(x) = \frac{1 + \log \alpha + \frac{1}{\eta}\log \frac{1}{z}}{\alpha^{\eta}z^{-1}}\le z\cdot\frac{2(1 + \log \alpha)\lb\frac{1}{\eta}\log \frac{1}{z}\vee 1\rb}{\alpha^{\eta}}.\]
Taking the logarithm of both sides, we obtain that 
\begin{align*}
\log g_{\eta}(x)  -\log z&\le \log (1 + \log \alpha) - \eta\log \alpha + \log \lb \frac{2}{\eta}\log \frac{1}{z}\vee 2\rb \\
& \le \log \lb \frac{2}{\eta}\log \frac{1}{z}\vee 2\rb - (\eta - 1)\log \alpha\\
& \le 0.
\end{align*}
\end{proof}

\section{One-Epoch Analysis}\label{app:one_epoch}
In order to apply Lemma \ref{lem:geom} on the sequence $\{D_{n}\}$, we need $\E |D_{N}| < \infty$. The following lemma justifies this property for several sequences that are involved in our later proofs. Its proof is distracting and relegated to the end of this section. 

\begin{lemma}\label{lem:geom_finite}
Assume that $\etaj L \le 1 / 4\Bj^{2/3}$. Then for any $j$, 
\[\E \|\td{x}_{j} - \td{x}_{j-1}\|^{2} < \infty, \quad \E (f(\td{x}_{j}) - f^{*}) < \infty, \quad \E \|\nabla f(\td{x}_{j})\|^{2} < \infty, \]
and
\[\E |\la \ej, \td{x}_{j} - \td{x}_{j-1}\ra| < \infty, \quad \E |\la \ej, \nabla f(\td{x}_{j})\ra| < \infty\]
\end{lemma}

As in the standard analysis of stochastic gradient methods, we start by establishing a bound on $\E_{\sI_{k}}\|\nuj_{k}\|^{2}$ and $\E_{\I_{j}}\|\ej\|^{2}$.
\begin{lemma}\label{lem:nuknorm}
Under Assumption \textbf{A}1,
\[\E_{\sI_{k}}\|\nuj_{k}\|^{2}\le \frac{L^{2}}{\bj}\|\xj_{k} - \xj_{0}\|^{2} + 2\|\nabla f(\xj_{k})\|^{2} + 2\|\ej\|^{2}.\]
\end{lemma}
\begin{proof}
  Using the fact that $\E \|Z\|^{2} = \E \|Z - \E Z\|^{2} + \|\E Z\|^{2}$ (for any random variable $Z$), we have
  \begin{align*}
& \E_{\sI_{k}}\|\nuj_{k}\|^{2}  = \E_{\sI_{k}}\|\nuj_{k} - \E_{\sI_{k}}\nuj_{k}\|^{2} + \|\E_{\sI_{k}}\nuj_{k}\|^{2}\\
 = & \E_{\sI_{k}}\|\nabla f_{\sI_{k}}(\xj_{k}) - \nabla f_{\sI_{k}}(\xj_{0}) - (\nabla f(\xj_{k}) - \nabla f(\xj_{0}))\|^{2} + \|\nabla f(\xj_{k}) + \ej\|^{2}\\
\le & \E_{\sI_{k}}\|\nabla f_{\sI_{k}}(\xj_{k}) - \nabla f_{\sI_{k}}(\xj_{0}) - (\nabla f(\xj_{k}) - \nabla f(\xj_{0}))\|^{2} + 2\|\nabla f(\xj_{k})\|^{2} + 2\|\ej\|^{2}.
\end{align*}
By Lemma \ref{lem:var_sampling},
\begin{align*}
  & \E_{\sI_{k}}\|\nabla f_{\sI_{k}}(\xj_{k}) - \nabla f_{\sI_{k}}(\xj_{0}) - (\nabla f(\xj_{k}) - \nabla f(\xj_{0}))\|^{2}\\
\le &\frac{1}{\bj}\cdot \frac{1}{n}\sum_{i=1}^{n}\|\nabla f_{i}(\xj_{k}) - \nabla f_{i}(\xj_{0}) - (\nabla f(\xj_{k}) - \nabla f(\xj_{0}))\|^{2} \\
= & \frac{1}{\bj}\cdot \lb\frac{1}{n}\sum_{i=1}^{n}\|\nabla f_{i}(\xj_{k}) - \nabla f_{i}(\xj_{0})\|^{2} - \|(\nabla f(\xj_{k}) - \nabla f(\xj_{0}))\|^{2}\rb \\
\le & \frac{1}{\bj}\cdot \frac{1}{n}\sum_{i=1}^{n}\|\nabla f_{i}(\xj_{k}) - \nabla f_{i}(\xj_{0})\|^{2}\\
\le & \frac{1}{\bj}\cdot L^{2}\|\xj_{k} - \xj_{0}\|^{2}
\end{align*}
where the last line uses Assumption \textbf{A}1. Therefore, 
\[\E_{\sI_{k}}\|\nuj_{k}\|^{2}\le \frac{L^{2}}{\bj}\|\xj_{k} - \xj_{0}\|^{2} + 2\|\nabla f(\xj_{k})\|^{2} + 2\|\ej\|^{2}.\]
\end{proof}

\begin{lemma}\label{lem:ejnorm}
  \[\E_{\I_{j}} \|\ej\|^{2} \le \frac{I(\Bj < n)}{\Bj}\cdot \H^{*}.\]
\end{lemma}
\begin{proof}
Since $\td{x}_{j-1}$ is independent of $\I_{j}$, conditioning on $\td{x}_{j-1}$ and applying Lemma \ref{lem:var_sampling}, we have
\begin{align*}
\E_{\I_{j}} \|\ej\|^{2} &= \frac{n - \Bj}{(n - 1)\Bj}\cdot \frac{1}{n}\sum_{i=1}^{n}\|\nabla f_{i}(\td{x}_{j-1}) - \nabla f(\td{x}_{j-1})\|^{2}\\
&\le \frac{n - \Bj}{(n - 1)\Bj}\cdot \H^{*}\le \frac{I(\Bj < n)}{\Bj}\cdot \H^{*}
\end{align*}
\end{proof}

Based on Lemma \ref{lem:geom}, Lemma \ref{lem:geom_finite}, Lemma \ref{lem:nuknorm} and Lemma \ref{lem:ejnorm}, we can derive bounds for primal and dual gaps respectively. 

\begin{lemma}\label{lem:primal}
Suppose $\eta_{j}L < 1$, then under Assumption \textbf{A}1,
\begin{align}
&\etaj\Bj(1 - \etaj L) \E\|\nabla f(\td{x}_{j})\|^{2} + \etaj\Bj\E \la\ej, \nabla f(\td{x}_{j})\ra\nonumber\\
\le &\bj\E\lb f(\td{x}_{j-1}) - f(\td{x}_{j})\rb + \frac{L^{3}\etaj^{2}\Bj}{2\bj}\E\|\td{x}_{j} - \td{x}_{j-1}\|^{2} + L\etaj^{2}\Bj\E \|\ej\|^{2}.\label{eq:primal}
\end{align}
where $\E$ denotes the expectation with respect to all randomness.
\end{lemma}

\begin{proof}
By \eqref{eq:A1},
\begin{align*}
&\E_{\sI_{k}}f(\xj_{k+1})\le f(\xj_{k}) -\etaj \la \E_{\sI_{k}}\nu_{k}, \nabla f(\xj_{k})\ra + \frac{L\etaj^{2}}{2}\E_{\sI_{k}}\|\nu_{k}\|^{2}\\
=&  f(\xj_{k}) -\etaj \|\nabla f(\xj_{k})\|^{2} - \etaj \la\ej, \nabla f(\xj_{k})\ra + \frac{L\etaj^{2}}{2}\E_{\sI_{k}}\|\nu_{k}\|^{2}\\
\le &f(\xj_{k}) - \etaj(1 - \etaj L) \|\nabla f(\xj_{k})\|^{2}- \etaj \la\ej, \nabla f(\xj_{k})\ra\\
&\quad  + \frac{L^{3}\etaj^{2}}{2\bj}\|\xj_{k} - \xj_{0}\|^{2} + L\etaj^{2}\|\ej\|^{2}  \quad (\mbox{Lemma \ref{lem:nuknorm}}).
\end{align*}
Let $\Ej$ denotes the expectation over $\sI_{0}, \sI_{1}, \ldots$, given $\Nj$. Note that $\Ej$ is equivalent to the expectation over $\sI_{0}, \sI_{1}, \ldots$ as $\Nj$ is independent of them. Since $\sI_{k+1}, \sI_{k+2}, \ldots$ are independent of $\xj_{k}$, the above inequality implies that 
\begin{align}
& \etaj(1 - \etaj L)\Ej\|\nabla f(\xj_{k})\|^{2} + \etaj\Ej \la\ej, \nabla f(\xj_{k})\ra\\
\le &\Ej f(\xj_{k}) - \Ej f(\xj_{k+1}) + \frac{L^{3}\etaj^{2}}{2\bj}\Ej\|\xj_{k} - \xj_{0}\|^{2} + L\etaj^{2}\|\ej\|^{2}.  \label{eq:primal1}
\end{align}
Let $k = \Nj$ in \eqref{eq:primal1}. By taking expectation with respect to $\Nj$ and using Fubini's theorem, we arrive at 
\begin{align}
&\etaj(1 - \etaj L) \E_{\Nj}\Ej\|\nabla f(\xj_{\Nj})\|^{2} + \etaj\E_{\Nj}\Ej \la\ej, \nabla f(\xj_{\Nj})\ra\nonumber\\
\le &\E_{\Nj}\lb\Ej f(\xj_{\Nj}) - \Ej f(\xj_{\Nj+1})\rb + \frac{L^{3}\etaj^{2}}{2\bj}\E_{\Nj}\Ej\E\|\xj_{\Nj} - \xj_{0}\|^{2} + L\etaj^{2}\|\ej\|^{2}\nonumber\\
 = & \frac{\bj}{\Bj}\lb f(\xj_{0}) - \Ej\E_{\Nj}f(\xj_{\Nj})\rb + \frac{L^{3}\etaj^{2}}{2\bj}\Ej\E_{\Nj}\|\xj_{\Nj} - \xj_{0}\|^{2} + L\etaj^{2}\|\ej\|^{2} \,\,\mbox{(Lemma \ref{lem:geom} and \ref{lem:geom_finite})}. \label{eq:primal2}
\end{align}
The lemma is then proved by substituting $\xj_{\Nj} (\xj_{0})$ by $\td{x}_{j} (\td{x}_{j-1})$, and taking an expectation over all past randomness.
\end{proof}

\begin{lemma}\label{lem:dual}
Suppose $\etaj^{2}L^{2}\Bj < \bj^{2}$, then under Assumption \textbf{A}1, 
\begin{align}
&\lb \bj - \frac{\etaj^{2}L^{2}\Bj}{\bj}\rb\E \|\td{x}_{j} - \td{x}_{j-1}\|^{2} + 2\etaj\Bj\E \la \ej, \td{x}_{j} - \td{x}_{j-1}\ra\nonumber\\
\le &-2\etaj\Bj\E \la\nabla f(\td{x}_{j}), \td{x}_{j} - \td{x}_{j-1}\ra + 2\etaj^{2}\Bj\E \|\nabla f(\td{x}_{j})\|^{2} + 2\etaj^{2}\Bj\E\|\ej\|^{2}.\label{eq:dual}
\end{align}
\end{lemma}

\begin{proof}
  Since $\xj_{k+1} = \xj_{k} - \etaj\nuj_{k}$, we have
  \begin{align*}
    & \E_{\sI_{k}}\|\xj_{k+1} - \xj_{0}\|^{2}\\ 
= & \|\xj_{k} - \xj_{0}\|^{2} - 2\etaj\la \E_{\sI_{k}}\nuj_{k}, \xj_{k} - \xj_{0}\ra + \etaj^{2}\E_{\sI_{k}}\|\nuj_{k}\|^{2}\\
= &\|\xj_{k} - \xj_{0}\|^{2} - 2\etaj\la\nabla f(\xj_{k}), \xj_{k} - \xj_{0}\ra - 2\etaj\la \ej, \xj_{k} - \xj_{0}\ra + \etaj^{2}\E_{\sI_{k}}\|\nuj_{k}\|^{2}\\
 \le & \lb 1 + \frac{\etaj^{2}L^{2}}{\bj}\rb\|\xj_{k} - \xj_{0}\|^{2} - 2\etaj\la\nabla f(\xj_{k}), \xj_{k} - \xj_{0}\ra - 2\etaj\la \ej, \xj_{k} - \xj_{0}\ra \\
& \quad + 2\etaj^{2}\|\nabla f(\xj_{k})\|^{2} + 2\etaj^{2}\|\ej\|^{2}. \quad \mbox{(Lemma \ref{lem:nuknorm})}
  \end{align*}
Using the same notation $\Ej$ as in the proof of Lemma \ref{lem:primal}, we have 
\begin{align}
&2\etaj\Ej\la\nabla f(\xj_{k}), \xj_{k} - \xj_{0}\ra + 2\etaj \Ej\la\ej, \xj_{k} - \xj_{0}\ra\nonumber\\
\le & \lb 1 + \frac{\etaj^{2}L^{2}}{\bj}\rb\Ej\|\xj_{k} - \xj_{0}\|^{2} - \Ej\|\xj_{k+1} - \xj_{0}\|^{2} + 2\etaj^{2}\|\nabla f(\xj_{k})\|^{2} + 2\etaj^{2}\|\ej\|^{2}.
  \label{eq:dual1}
\end{align}
Let $k = \Nj$ in \eqref{eq:dual1}. By taking expectation with respect to $\Nj$ and using Fubini's theorem, we arrive at 
\begin{align}
&2\etaj\E_{\Nj}\Ej\la\nabla f(\xj_{\Nj}), \xj_{\Nj} - \xj_{0}\ra + 2\etaj\E_{\Nj} \Ej\la\ej, \xj_{\Nj} - \xj_{0}\ra\nonumber\\
\le & \lb 1 + \frac{\etaj^{2}L^{2}}{\bj}\rb\E_{\Nj}\Ej\|\xj_{\Nj} - \xj_{0}\|^{2} - \E_{\Nj}\Ej\|\xj_{\Nj+1} - \xj_{0}\|^{2} + 2\etaj^{2}\E_{\Nj}\|\nabla f(\xj_{\Nj})\|^{2} + 2\etaj^{2}\|\ej\|^{2}\nonumber\\
= & \lb-\frac{\bj}{\Bj} + \frac{\etaj^{2}L^{2}}{\bj}\rb\E_{\Nj}\Ej\|\xj_{\Nj} - \xj_{0}\|^{2} + 2\etaj^{2}\E_{\Nj}\|\nabla f(\xj_{\Nj})\|^{2} + 2\etaj^{2}\|\ej\|^{2}. \,\, \mbox{(Lemma \ref{lem:geom} and \ref{lem:geom_finite})}\label{eq:dual2}
\end{align}
The lemma is then proved by substituting $\xj_{\Nj} (\xj_{0})$ by $\td{x}_{j} (\td{x}_{j-1})$ and taking a further expectation with respect to the past randomness.
\end{proof}

\begin{lemma}\label{lem:Ij}
\begin{equation}\label{eq:Ij}
\bj\E \la\ej, \td{x}_{j} - \td{x}_{j-1}\ra = -\etaj\Bj\E \la\ej, \nabla f(\td{x}_{j})\ra - \etaj\Bj\E\|\ej\|^{2}.
\end{equation}
\end{lemma}
\begin{proof}
  Let $\Mj_{k} = \la \ej, \xj_{k} - \xj_{0}\ra$. By definition, we have
\[\E_{\Nj} \la\ej, \td{x}_{j} - \td{x}_{j-1}\ra = \E_{\Nj}\Mj_{\Nj}.\]
Since $\Nj$ is independent of $(\xj_{0}, \ej)$, this implies that 
\[\E \la\ej, \td{x}_{j} - \td{x}_{j-1}\ra = \E \Mj_{\Nj}.\]
Also we have $\Mj_{0} = 0$. On the other hand,
\begin{align*}
\E_{\sI_{k}}\lb\Mj_{k+1} - \Mj_{k}\rb &= \E_{\sI_{k}}\la \ej, \xj_{k+1} - \xj_{k}\ra = -\etaj\la\ej, \E_{\sI_{k}}\nuj_{k}\ra\\ 
& = -\etaj\la\ej, \nabla f(\xj_{k})\ra - \etaj\|\ej\|^{2}.
\end{align*}
Using the same notation $\Ej$ as in the proof of Lemma \ref{lem:primal} and Lemma \ref{lem:dual}, we have
\begin{equation}\label{eq:Ij1}
\Ej\lb\Mj_{k+1} - \Mj_{k}\rb = -\etaj\la\ej, \Ej\nabla f(\xj_{k})\ra - \etaj\|\ej\|^{2}.
\end{equation}
Let $k = \Nj$ in \eqref{eq:Ij1}. By taking an expectation with respect to $\Nj$ and using Lemma \ref{lem:geom} and \ref{lem:geom_finite}, we obtain that 
\[\frac{\bj}{\Bj}\E_{\Nj}\Mj_{\Nj} = -\etaj\la\ej, \E_{\Nj}\Ej\nabla f(\xj_{\Nj})\ra - \etaj\|\ej\|^{2}.\]
The lemma is then proved by substituting $\xj_{\Nj} (\xj_{0})$ by $\td{x}_{j} (\td{x}_{j-1})$ and taking a further expectation with respect to the past randomness.
\end{proof}

\begin{proof}[\textbf{Theorem \ref{thm:one_epoch}}]
Multiplying equation \eqref{eq:primal} by $2$, equation \eqref{eq:dual} by $\frac{\bj}{\etaj\Bj}$ and summing them, we obtain that
\begin{align}
&2\etaj\Bj(1 - \etaj L - \frac{\bj}{\Bj}) \E\|\nabla f(\td{x}_{j})\|^{2} + \frac{\bj^{3} - \etaj^{2}L^{2}\bj\Bj - \etaj^{3}L^{3}\Bj^{2}}{\bj\etaj\Bj}\E \|\td{x}_{j} - \td{x}_{j-1}\|^{2}  \nonumber\\
& + 2\etaj\Bj \E \la\ej, \nabla f(\td{x}_{j})\ra + 2\bj\E \la\ej, \td{x}_{j} - \td{x}_{j-1}\ra \nonumber\\
 \le & - 2\bj\E \la\nabla f(\td{x}_{j}), \td{x}_{j} - \td{x}_{j-1}\ra + 2\bj\E (f(\td{x}_{j-1}) - f(\td{x}_{j}))  + \lb 2L\etaj^{2}\Bj + 2\etaj\bj\rb\E \|\ej\|^{2}.\label{eq:final1}
\end{align}
By Lemma \ref{lem:Ij}, the second row can be simplified as
\[2\etaj\Bj \E \la\ej, \nabla f(\td{x}_{j})\ra + 2\bj\E \la\ej, \td{x}_{j} - \td{x}_{j-1}\ra = -2\etaj\Bj\E \|\ej\|^{2}.\]
Using the fact that $2\la a, b\ra\le \beta\|a\|^{2} + \frac{1}{\beta}\|b\|^{2}$ for any $\beta> 0$, we have
\begin{align*}
&- 2\bj\E \la\nabla f(\td{x}_{j}), \td{x}_{j} - \td{x}_{j-1}\ra \\
  \le &\frac{\bj\etaj\Bj}{\bj^{3} - \etaj^{2}L^{2}\bj\Bj - \etaj^{3}L^{3}\Bj^{2}}\cdot \bj^{2}\E \|\nabla f(\td{x}_{j})\|^{2} + \frac{\bj^{3} - \etaj^{2}L^{2}\bj\Bj - \etaj^{3}L^{3}\Bj^{2}}{\bj\etaj\Bj}\E\|\td{x}_{j} - \td{x}_{j-1}\|^{2}.
\end{align*}
Putting the pieces together, we conclude that 
\begin{align}
  &\frac{\etaj\Bj}{\bj}\lb 2 - \frac{2\bj}{\Bj} - 2\etaj L - \frac{\bj^{3}}{\bj^{3} - \etaj^{2}L^{2}\bj\Bj - \etaj^{3}L^{3}\Bj^{2}}\rb\E \|\nabla f(\td{x}_{j})\|^{2}\nonumber\\
\le & 2\E (f(\td{x}_{j-1}) - f(\td{x}_{j}))  + \frac{2\etaj\Bj}{\bj}\lb 1 + \etaj L + \frac{\bj}{\Bj}\rb\E \|\ej\|^{2}.\label{eq:final2}
\end{align}
Since $\etaj L = \thetaj = \gamma (\bj / \Bj)^{\frac{2}{3}}$ and $\bj \ge 1, \Bj\ge 8\bj\ge 8$,
\[\bj^{3} - \etaj^{2}L^{2}\bj\Bj - \etaj^{3}L^{3}\Bj^{2} = b_j^{3}\lb 1 - \gamma^{2}\cdot \bj^{-\frac{2}{3}}\Bj^{-\frac{1}{3}} - \gamma^{3}\cdot \bj^{-1}\rb\ge \bj^{3}(1 - \gamma^{2}/2 - \gamma^{3}).\]
Then \eqref{eq:final2} can be simplified as
\begin{align}
  &\gamma\lb\frac{\Bj}{\bj}\rb^{\frac{1}{3}}\lb 2 - \frac{2\bj}{\Bj} - 2\gamma \lb\frac{\bj}{\Bj}\rb^{\frac{2}{3}} - \frac{1}{1 - \gamma^{2} / 2 - \gamma^{3}}\rb\E \|\nabla f(\td{x}_{j})\|^{2}\nonumber\\
  \le & 2L\E (f(\td{x}_{j-1}) - f(\td{x}_{j}))  + 2\gamma\lb 1 + \gamma \lb\frac{\bj}{\Bj}\rb^{\frac{2}{3}}+ \frac{\bj}{\Bj}\rb \lb\frac{\Bj}{\bj}\rb^{\frac{1}{3}}\E \|\ej\|^{2}\nonumber\\
  \le & 2L\E (f(\td{x}_{j-1}) - f(\td{x}_{j}))  + 2\gamma\lb 1 + \gamma \lb\frac{\bj}{\Bj}\rb^{\frac{2}{3}}+ \frac{\bj}{\Bj}\rb \frac{I(\Bj < n)}{\bj^{\frac{1}{3}}\Bj^{\frac{2}{3}}}\cdot\H^{*}\,\, (\mbox{By Lemma \ref{lem:ejnorm}})\label{eq:final3}.
\end{align}
Since $\Bj \ge 8\bj, \gamma \le \frac{1}{3}$, we have
\[2 - \frac{2\bj}{\Bj} - 2\gamma \lb\frac{\bj}{\Bj}\rb^{\frac{2}{3}} - \frac{1}{1 - \gamma^{2} / 2 - \gamma^{3}}\ge 2 - \frac{1}{4} - \frac{\gamma}{2} - \frac{1}{1 - \gamma^{2} / 2 - \gamma^{3}}\ge 0.482\]
and 
\[1 + \gamma \lb\frac{\bj}{\Bj}\rb^{\frac{2}{3}}+ \frac{\bj}{\Bj}\le 1 + \frac{\gamma}{4} + \frac{1}{8}\le 1.209.\]
Thus, \eqref{eq:final3} implies that 
\begin{equation}
  \label{eq:finalfinal}
  \lb\frac{\Bj}{\bj}\rb^{\frac{1}{3}}\E \|\nabla f(\td{x}_{j})\|^{2}\le \frac{5L}{\gamma}\cdot \E (f(\td{x}_{j-1}) - f(\td{x}_{j})) + \frac{6I(\Bj < n)}{\bj^{\frac{1}{3}}\Bj^{\frac{2}{3}}}\cdot\H^{*}.
\end{equation}

\end{proof}

\begin{proof}[\textbf{of Lemma \ref{lem:geom_finite}}]
We prove the first two claims by induction. The first inequality in the proof of Lemma \ref{lem:primal} yields 
\begin{align*}
&\etaj(1 - \etaj L) \|\nabla f(\xj_{k})\|^{2} \\
\le &f(\xj_{k}) - \E_{\sI_{k}}f(\xj_{k+1}) - \etaj \la\ej, \nabla f(\xj_{k})\ra  + \frac{L^{3}\etaj^{2}}{2\bj}\|\xj_{k} - \xj_{0}\|^{2} + L\etaj^{2}\|\ej\|^{2}.
\end{align*}
Using the fact that $2\la a, b\ra = \|a\|^{2} / c + c\|b\|^{2} - \|a - cb\|^{2} / c$ for any $c > 0$, we have
\[- \etaj \la\ej, \nabla f(\xj_{k})\ra\le \frac{1}{4}\etaj\|\nabla f(\xj_{k})\|^{2} + \etaj\|\ej\|^{2}.\]
Since $\etaj L \le 1 / 4\Bj^{2/3}\le 1/ 4$, we obtain that
\begin{align}
  \etaj \|\nabla f(\xj_{k})\|^{2}&\le 2\lb f(\xj_{k}) - \E_{\sI_{k}}f(\xj_{k+1})\rb   + \frac{L^{3}\etaj^{2}}{\bj}\|\xj_{k} - \xj_{0}\|^{2} + 2(L\etaj^{2} + \etaj)\|\ej\|^{2}\nonumber\\
& \le 2\lb f(\xj_{k}) - \E_{\sI_{k}}f(\xj_{k+1})\rb   + \frac{L^{3}\etaj^{2}}{\bj}\|\xj_{k} - \xj_{0}\|^{2} + 3\etaj\|\ej\|^{2}\label{eq:geom_finite_1}.
\end{align}
On the other hand, the first inequality in the proof of Lemma \ref{lem:dual} implies that 
  \begin{align*}
    & \E_{\sI_{k}}\|\xj_{k+1} - \xj_{0}\|^{2}\\ 
 \le & \lb 1 + \frac{\etaj^{2}L^{2}}{\bj}\rb\|\xj_{k} - \xj_{0}\|^{2} - 2\etaj\la\nabla f(\xj_{k}), \xj_{k} - \xj_{0}\ra - 2\etaj\la \ej, \xj_{k} - \xj_{0}\ra \\
& \quad + 2\etaj^{2}\|\nabla f(\xj_{k})\|^{2} + 2\etaj^{2}\|\ej\|^{2}.
\end{align*}
Using the fact that $2\la a, b\ra = \|a\|^{2} / c + c\|b\|^{2} - \|a - cb\|^{2} / c$ for any $c > 0$, we have
\[-2\etaj\la\nabla f(\xj_{k}), \xj_{k} - \xj_{0}\ra \le 8\etaj^{2}\Bj \|\nabla f(\xj_{k})\|^{2} + \frac{1}{8\Bj}\|\xj_{k} - \xj_{0}\|^{2},\]
and 
\[-2\etaj\la\ej, \xj_{k} - \xj_{0}\ra \le 8\etaj^{2}\Bj \|\ej\|^{2} + \frac{1}{8\Bj}\|\xj_{k} - \xj_{0}\|^{2}.\]
Thus, 
\begin{align}
    & \E_{\sI_{k}}\|\xj_{k+1} - \xj_{0}\|^{2}\nonumber\\ 
 \le & \lb 1 + \frac{\etaj^{2}L^{2}}{\bj} + \frac{1}{4\Bj}\rb\|\xj_{k} - \xj_{0}\|^{2} + (2\etaj^{2} + 8\etaj^{2}\Bj)\|\nabla f(\xj_{k})\|^{2} + (2\etaj^{2} + 8\etaj^{2}\Bj)\|\ej\|^{2}\nonumber\\
\le & \lb 1 + \frac{5}{16\Bj}\rb\|\xj_{k} - \xj_{0}\|^{2} + 10\etaj^{2}\Bj\|\nabla f(\xj_{k})\|^{2} + 10\etaj^{2}\Bj\|\ej\|^{2}, \label{eq:geom_finite_2}
\end{align}
where the last line uses the condition that $\etaj L \le 1 / 4$. Plugging \eqref{eq:geom_finite_1} into \eqref{eq:geom_finite_2}, we obtain that
\begin{align}
    & \E_{\sI_{k}}\|\xj_{k+1} - \xj_{0}\|^{2}\le  \lb 1 + \frac{5}{16\Bj}\rb\|\xj_{k} - \xj_{0}\|^{2} + 10\etaj^{2}\Bj\|\ej\|^{2}\nonumber\\
& \quad + 20\etaj \Bj (f(\xj_{k}) - \E_{\sI_{k}}f(\xj_{k+1})) + \frac{10(\etaj L)^{3}\Bj}{\bj}\|\xj_{k} - \xj_{0}\|^{2} + 40\etaj^{2}\Bj \|\ej\|^{2}\nonumber\\
& \le \lb 1 + \frac{5}{16\Bj} + \frac{10}{64\Bj}\rb\|\xj_{k} - \xj_{0}\|^{2} + 20\etaj \Bj (f(\xj_{k}) - \E_{\sI_{k}}f(\xj_{k+1})) + 40\etaj^{2}\Bj \|\ej\|^{2}\nonumber\\
& \le \lb 1 + \frac{1}{2\Bj}\rb \|\xj_{k} - \xj_{0}\|^{2} + 20\etaj \Bj (f(\xj_{k}) - \E_{\sI_{k}}f(\xj_{k+1})) + 40\etaj^{2}\Bj \|\ej\|^{2}.\label{eq:geom_finite_3}
\end{align}
Let 
\[\Lj_{k} = 20\etaj \Bj \E (f(\xj_{k}) - f^{*}) + \E \|\xj_{k} - \xj_{0}\|^{2}.\]
Taking expectation in \eqref{eq:geom_finite_3} then entails that
\begin{align*}
  &\Lj_{k+1} \le \lb 1 + \frac{1}{2\Bj}\rb \Lj_{k} + 40\etaj^{2}\Bj \E\|\ej\|^{2}\\
\Longleftrightarrow &\Lj_{k+1} + 80\etaj^{2}\Bj^{2}\E\|\ej\|^{2} \le \lb 1 + \frac{1}{2\Bj}\rb \lb \Lj_{k} + 80\etaj^{2}\Bj^{2}\E\|\ej\|^{2}\rb\\
\Longrightarrow & \Lj_{k}\le \lb 1 + \frac{1}{2\Bj}\rb^{k}\lb\Lj_{0} + 80\etaj^{2}\Bj^{2}\E\|\ej\|^{2}\rb.
\end{align*}
Recall that $N_{j}\sim \mathrm{Geom}\lb B_{j}/(B_{j} + b_{j})\rb$, i.e.
\[P(N_{j} = k) = \frac{b_{j}}{B_{j} + b_{j}}\lb\frac{B_{j}}{B_{j} + b_{j}}\rb^{k}\le \lb\frac{B_{j}}{B_{j} + 1}\rb^{k}.\]
Then
\[\E \lb 1 + \frac{1}{2B_{j}}\rb^{N_{j}} \le \sum_{k\ge 0}\lb\frac{2B_{j} + 1}{2B_{j}}\frac{B_{j}}{B_{j} + 1}\rb^{k} = \sum_{k\ge 0}\lb\frac{2B_{j} + 1}{2B_{j} + 2}\rb^{k} = 2B_{j} + 2\]
This together with the induction hypothesis entail that 
\[\E \Lj_{\Nj} \le (2B_{j} + 2)\lb \Lj_{0} + 80\etaj^{2}\Bj^{2}\E\|\ej\|^{2}\rb.\]
By Lemma \ref{lem:ejnorm}, $\E \|\ej\|^{2} < \infty$. Further by the induction hypothesis, we prove that
\[\E \Lj_{N_{j}} < \infty\]
and hence the first two claims.

~\\
\noindent To other claims are simple consequences of the first two. In fact, the third claim is followed by \eqref{eq:geom_finite_1} and the last two claims followed by the first two claims and the fact that $\la a, b\ra\le \|a\|^{2} / 2 + \|b\|^{2} / 2$.
\end{proof}

\section{Convergence Analysis for Smooth Objectives}\label{app:ana_smooth}
\begin{proof}[\textbf{Theorem \ref{thm:smooth}}]
Since $\td{x}_{T}^{*}$ is a random element from $(\td{x}_{j})_{j=1}^{T}$ with 
\[P(\td{x}_{T}^{*} = \td{x}_{j}) \propto \frac{\etaj\Bj}{\bj} \propto \lb \Bj/\bj\rb^{\frac{1}{3}},\]
we have
\begin{align*}
\E \|\nabla f(\td{x}_{T}^{*})\|^{2}&\le \frac{\frac{5L}{\gamma}\cdot \E (f(\td{x}_{0}) - f(\td{x}_{T+1})) + 6\lb\sum_{j=1}^{T}\bj^{-\frac{1}{3}}\Bj^{-\frac{2}{3}}I(\Bj < n)\rb\H^{*}}{\sum_{j=1}^{T}\bj^{-\frac{1}{3}}\Bj^{\frac{1}{3}}}\\
&\le \frac{\frac{5L}{\gamma}\cdot (f(\td{x}_{0}) - f^{*}) + 6\lb\sum_{j=1}^{T}\bj^{-\frac{1}{3}}\Bj^{-\frac{2}{3}}I(\Bj < n)\rb\H^{*}}{\sum_{j=1}^{T}\bj^{-\frac{1}{3}}\Bj^{\frac{1}{3}}}.
\end{align*}
\end{proof}

\begin{proof}[\textbf{Corollary \ref{cor:smooth_constant}}]
By Theorem \ref{thm:smooth},
\[\E \|\nabla f(\td{x}_{T}^{*})\|^{2}\le \frac{30\Delta_{f} + 6 TB^{-\frac{2}{3}}I(B < n)\cdot \H^{*}}{TB^{\frac{1}{3}}} = \frac{30\Delta_{f}}{TB^{\frac{1}{3}}} + \frac{6\H^{*}\cdot I(B < n)}{B}.\]
Let $\td{T}(\eps)$ be the minimum number of epochs such that 
\[\frac{30\Delta_{f}}{\td{T}(\eps)B^{\frac{1}{3}}}\le \frac{\eps}{2}.\]
Then under the setting of the corollary, for any $T\ge \td{T}(\eps)$,
\[\E \|\nabla f(\td{x}_{T}^{*})\|^{2}\le \frac{\eps}{2} + \frac{6\H^{*}\cdot I(B < n)}{B}\le \frac{\eps}{2}\le \eps.\]
By definition, we know that $T(\eps)\le \td{T}(\eps)$. Noticing that 
\[\td{T}(\eps) = O\lb\left\lceil\frac{\Delta_{f}}{\eps B^{\frac{1}{3}}}\right\rceil\rb = O\lb 1 + \frac{\Delta_{f}}{\eps B^{\frac{1}{3}}}\rb,\]
we conclude that 
\[\E \comp(\eps) = O(T(\eps)B) = O(\td{T}(\eps)B) = O\lb B + \frac{\Delta_{f}}{\eps}\cdot B^{\frac{2}{3}}\rb.\]
The corollary is then proved by substituting for $B$.
\end{proof}

\begin{proof}[\textbf{Corollary \ref{cor:smooth_vary}}]
By Theorem \ref{thm:smooth},
\[\E \|\nabla f(\td{x}_{T}^{*})\|^{2}\le \frac{30\Delta_{f} + 6 \sum_{j=1}^{T}\frac{I(\Bj < n)}{j}\H^{*}}{\sum_{j=1}^{T}(j^{\frac{1}{2}}\wedge n^{\frac{1}{3}})}\triangleq W(T).\]  
Let $T_{*} = \lfloor n^{\frac{2}{3}} \rfloor$. First we prove that $W(T)$ is strictly decreasing. 
\begin{enumerate}
\item When $T \ge T_{*}$, the numerator is a constant and the denominator is strictly increasing. Thus, $W(T)$ is strictly decreasing on $[T_{*}, \infty)$;
\item When $T < T_{*}$, let $a_{1j} = \frac{6\H^{*}}{j}$ and $a_{2j} = j^{\frac{1}{2}}$. Further let 
\[U(T) = \frac{30\Delta_{f}}{\sum_{j=1}^{T}a_{2j}}, \quad V(T) = \frac{\sum_{j=1}^{T}a_{1j}}{\sum_{j=1}^{T}a_{2j}},\]
then 
\[W(T) = U(T) + V(T).\]
It is obvious that $U(T)$ is strictly decreasing. Noticing that $\frac{a_{1j}}{a_{2j}} = \frac{6\H^{*}}{j^{\frac{3}{2}}}$ is strictly decreasing, we also conclude that $V(T)$ is strictly decreasing. Therefore, $W(T)$ is strictly decreasing on $[1, T_{*}]$.
\end{enumerate}
In summary, $W(T)$ is strictly decreasing. 

Now we derive the bound for $T(\eps)$. To do so, we distinguish two cases to analyze $W(T)$. 
\begin{enumerate}
\item If $T \le T_{*}$, then 
\[W(T) = \frac{30\Delta_{f} + 6 \lb\sum_{j=1}^{T}\frac{1}{j}\rb\H^{*}}{\sum_{j=1}^{T}j^{\frac{1}{2}}}.\]
Since $\frac{1}{j}$ is decreasing, we have
\[\sum_{j=1}^{T}\frac{1}{j} = 1 + \sum_{j=2}^{T}\frac{1}{j}\le 1 + \int_{1}^{T}\frac{dx}{x} = 1 + \log T.\]
Similarly, since $j^{\frac{1}{2}}$ is increasing, we have
\[\sum_{j=1}^{T}j^{\frac{1}{2}}\ge \int_{0}^{T}x^{\frac{1}{2}}dx = \frac{2}{3}T^{\frac{3}{2}}.\]
Therefore,
\[W(T)\le \frac{45\Delta_{f} + 9 (1 + \log T)\H^{*}}{ T^{\frac{3}{2}}}.\]
\item If $T > T_{*}$, then
\begin{equation}\label{eq:WT_case2}
W(T) = \frac{30\Delta_{f} + 6 \lb\sum_{j=1}^{T_{*}}\frac{1}{j}\rb\H^{*}}{\sum_{j=1}^{T_{*}}j^{\frac{1}{2}} + n^{\frac{1}{3}}(T - T_{*})} = W(T_{*})\cdot \frac{\sum_{j=1}^{T_{*}}j^{\frac{1}{2}}}{\sum_{j=1}^{T_{*}}j^{\frac{1}{2}} + n^{\frac{1}{3}}(T - T_{*})}.
\end{equation}
Similar to the first case, we have
\[\sum_{j=1}^{T_{*}}j^{\frac{1}{2}} = \sum_{j=1}^{T_{*}-1}j^{\frac{1}{2}} + T_{*}^{\frac{1}{2}}\le \int_{1}^{T_{*}}\sqrt{x}dx + n^{\frac{1}{3}} = \frac{2}{3}n + n^{\frac{1}{3}} - \frac{2}{3}.\]
Since 
\[n^{\frac{1}{3}} - \frac{2}{3} = \frac{n - 1}{n^{\frac{2}{3}} + n^{\frac{1}{3}} + 1} + \frac{1}{3}\le \frac{n}{3},\]
we obtain that 
\[\sum_{j=1}^{T_{*}}j^{\frac{1}{2}}\le n.\]
As a result, \eqref{eq:WT_case2} implies that 
\[W(T)\le W(T_{*})\cdot \frac{1}{1 + n^{-\frac{2}{3}}(T - T_{*})}.\]
\end{enumerate}

Putting the pieces together, we obtain that 
\begin{equation}\label{eq:WT_bound}
W(T)\le \left\{
    \begin{array}{lll}
\displaystyle\frac{45\Delta_{f} + 9 (1 + \log T)\H^{*}}{T^{\frac{3}{2}}}& \triangleq W_{1}(T)& (T \le T_{*})\\
  \displaystyle \frac{W(T_{*})}{1 + n^{-\frac{2}{3}}(T - T_{*})} & \triangleq W_{2}(T)& (T > T_{*}).
    \end{array}
\right.
\end{equation}
It is easy to see that both $W_{1}(T)$ and $W_{2}(T)$ are strictly decreasing and $\lim_{T\rightarrow \infty}W_{1}(T) = \lim_{T\rightarrow \infty}W_{2}(T) = 0$. Let
\[T_{1}(\eps) = \min\{T: W_{1}(T)\le \eps\}, \quad T_{2}(\eps) = \min\{T\ge T_{*}: W_{2}(T)\le \eps\}.\]
Recall that $W(T)$ is also strictly decreasing, we have 
\[T(\eps)\le \left\{
    \begin{array}{cc}
      T_{1}(\eps) & (W(T_{*}) \le \eps)\\
      T_{2}(\eps) & (W(T_{*}) > \eps).
    \end{array}
\right.\]
More concisely, 
\begin{equation}
  \label{eq:Teps}
  T(\eps)\le T_{1}(\eps)\wedge T_{*} + (T_{2}(\eps) - T_{*}).
\end{equation}
To derive a bound for $T_{1}(\eps)$, let $\td{T}_{1}(\eps)$ be the minimum $T$ such that 
\[\frac{45\Delta_{f}}{T^{\frac{3}{2}}}\le \frac{\eps}{2}, \quad \frac{9\H^{*}(1 + \log T)}{T^{\frac{3}{2}}}\le \frac{\eps}{2}.\]
Then by Lemma \ref{lem:log}, we have
\[T_{1}(\eps)\le \td{T}_{1}(\eps) = O\lb\lb\frac{\Delta_{f}}{\eps}\rb^{\frac{2}{3}} + \lb\frac{\H^{*}}{\eps}\rb^{\frac{2}{3}}\cdot \log^{2} \lb\frac{\H^{*}}{\eps}\rb\rb.\]
On the other hand, it is straightforward to derive a bound for  $T_{2}(\eps)$ as 
\[T_{2}(\eps) -  T_{*} \le \lb n^{\frac{2}{3}}\cdot \frac{W(T_{*}) - \eps}{\eps}\rb_{+}\le n^{\frac{2}{3}}\cdot \frac{W(T_{*})}{\eps} = O\lb \frac{\Delta_{f}}{\eps n^{\frac{1}{3}}} + \frac{\H^{*}\log n}{\eps n^{\frac{1}{3}}}\rb.\]
Therefore, we conclude that
\begin{equation}\label{eq:Tepsfinal}
T(\eps) = O\lb \min\left\{\frac{1}{\eps^{\frac{2}{3}}}\left[\Delta_{f}^{\frac{2}{3}} + (\H^{*})^{\frac{2}{3}}\log^{2} \lb\frac{\H^{*}}{\eps}\rb\right], n^{\frac{2}{3}}\right\} + \frac{\Delta_{f} + \H^{*}\log n}{\eps n^{\frac{1}{3}}}\rb.
\end{equation}

Finally, to obtain the bound for the computation complexity, we notice that 
\[\sum_{j=1}^{T}j^{\frac{3}{2}} = O(T^{\frac{5}{2}}).\]
Let $x_{+}$ denote the positive part of $x$; i.e., $x_{+} = \max\{x, 0\}$. Therefore, 
\begin{align*}
\E \comp(\eps) &= O\lb\sum_{j=1}^{T(\eps)}B_{j}\rb = O\lb (T(\eps)\wedge T_{*})^{\frac{5}{2}} + n(T(\eps) - T_{*})_{+}\rb\\
& = O\lb (T_{1}(\eps)\wedge T_{*})^{\frac{5}{2}} + n(T_{2}(\eps) - T_{*})\rb\\
& = O\lb \min\left\{\frac{1}{\eps^{\frac{5}{3}}}\left[\Delta_{f}^{\frac{5}{3}} + (\H^{*})^{\frac{5}{3}}\log^{5} \lb\frac{\H^{*}}{\eps}\rb\right], n^{\frac{5}{3}}\right\} + \frac{n^{\frac{2}{3}}}{\eps}\cdot (\Delta_{f} + \H^{*}\log n)\rb.
\end{align*}
\end{proof}

\begin{remark}\label{rem:smooth_log}
The log-factor $\log^{5}\lb\frac{1}{\eps}\rb$ can be reduced to $\log^{\frac{3}{2} + \mu}\lb\frac{1}{\eps}\rb$ for any $\mu > 0$ by setting $\Bj =\lceil j^{\frac{3}{2}}(\log j)^{\frac{3}{2} + \mu}\wedge n\rceil$. In this case,
\[W(T) = \frac{30\Delta_{f} + 6\lb\sum_{j=1}^{T}\frac{I(\Bj < n)}{j(\log j)^{1 + \frac{2\mu}{3}}}\rb\H^{*}}{\sum_{j=1}^{T}j^{\frac{1}{2}}(\log j)^{\frac{1}{2} + \frac{\mu}{3}}}.\]
For any $\mu > 0$, 
\[\sum_{j=1}^{T}\frac{I(\Bj < n)}{j(\log j)^{1 + \frac{2\mu}{3}}}\le 1 + \int_{1}^{\infty}\frac{1}{x(\log x)^{1 + \frac{2\mu}{3}}} < \infty.\]
On the other hand, as proved above,
\[\sum_{j=1}^{T}j^{\frac{1}{2}}(\log j)^{\frac{1}{2} + \frac{\mu}{3}}\ge \sum_{j=1}^{T}j^{\frac{1}{2}}\sim T^{\frac{3}{2}}.\]
Thus, 
\[W(T)\sim O\lb\frac{\Delta_{f} + \H^{*}}{T^{\frac{3}{2}}}\rb.\]
Using similar arguments and treating $\Delta_{f}$ and $\H^{*}$ as $O(1)$ for simplicity, we can obtain that 
\[T(\eps) = O\lb\eps^{-\frac{2}{3}}\wedge n^{\frac{2}{3}} + \frac{1}{\eps n^{\frac{1}{3}}}\rb.\]
If $B_{T(\eps)} < n$, then 
\[\E\comp(\eps) = O\lb\sum_{j=1}^{T(\eps)}j^{\frac{3}{2}}(\log j)^{\frac{3}{2} + \mu}\rb = O\lb T(\eps)^{\frac{5}{2}}\cdot (\log T(\eps))^{\frac{3}{2} + \mu}\rb = O\lb\eps^{-\frac{5}{3}}\log^{\frac{3}{2} + \mu}\lb\frac{1}{\eps}\rb\rb.\]
If $B_{T(\eps)} \ge n$, we obtain the same bound as in Corollary \ref{cor:smooth_vary}.
\end{remark}

\section{Convergence Analysis for P-L Objectives}\label{app:ana_PL}
\begin{proof}[\textbf{Theorem \ref{thm:PL}}]
By equation \eqref{eq:finalfinal} in the proof of Theorem \ref{thm:smooth} (see p.~\pageref{eq:finalfinal}) and the P-L condition,
\begin{align*}
&\mu \lb\frac{\Bj}{\bj}\rb^{\frac{1}{3}}\E(f(\td{x}_{j}) - f^{*}) \le \lb\frac{\Bj}{\bj}\rb^{\frac{1}{3}}\E \|\nabla f(\td{x}_{j})\|^{2}\\
\le & \frac{5L}{\gamma}\cdot \E (f(\td{x}_{j-1}) - f(\td{x}_{j})) + 6\bj^{-\frac{1}{3}}\Bj^{-\frac{2}{3}}I(\Bj < n)\cdot\H^{*}.  
\end{align*}
For brevity, we write $F_{j}$ for $\E (f(\td{x}_{j}) - f^{*})$. Then 
\begin{equation}\label{eq:PL1}
\lb\mu \gamma\Bj^{\frac{1}{3}} + 5L\bj^{\frac{1}{3}}\rb F_{j}\le 5L\bj^{\frac{1}{3}} F_{j-1} + 6\gamma\Bj^{-\frac{2}{3}}I(\Bj < n)\cdot\H^{*}.
\end{equation}
By definition of $\lambda_{j}$, this can be reformulated as
\[F_{j}\le \lambda_{j}F_{j-1} + 6\gamma\H^{*}\cdot \frac{I(\Bj < n)}{\mu\gamma \Bj + 5L\bj^{\frac{1}{3}}\Bj^{\frac{2}{3}}}.\]
Applying the above inequality iteratively for $j= T, T-1, \ldots, 1$, we prove the result.
\end{proof}

\begin{proof}[\textbf{Corollary \ref{cor:PL_constant}}]
When $\Bj \equiv B, \bj \equiv 1$ and $\gamma = \frac{1}{6}$, \eqref{eq:PL1} in the proof of Theorem \ref{thm:PL} can be reformulated as 
\[\lb\mu\gamma B^{\frac{1}{3}} + 30L\rb \lb F_{j} - \frac{6\H^{*}I( B  < n)}{\mu B }\rb \le 30L \lb F_{j-1} - \frac{6\H^{*}I( B  < n)}{\mu B }\rb.\]
This implies that 
\[F_{T}\le \lb\frac{30L}{\mu B^{\frac{1}{3}} + 30L}\rb^{T}\Delta_{f} + \frac{6\H^{*}I(B < n)}{\mu B}.\]
Under the setting of this problem, 
\[\frac{6\H^{*}I(B < n)}{\mu B}\le \frac{\eps}{2}.\]
By definition of $T(\eps)$, we have
\[T(\eps)\le \log \frac{\Delta_{f}}{\eps} \bigg/\log \lb\frac{30L}{\mu B^{\frac{1}{3}} + 30L}\rb = O\lb \log \frac{\Delta_{f}}{\eps}\lb 1 + \frac{L}{\mu B^{\frac{1}{3}}}\rb\rb.\]
As a consequence, 
\[\E \comp(\eps) = O\lb T(\eps) B\rb = O\lb \lb B + \frac{L B^{\frac{2}{3}}}{\mu}\rb\log \frac{\Delta_{f}}{\eps}\rb.\]
Plugging into $B$, we end up with
\[\E \comp(\eps) = O\lb \left\{\lb\frac{\H^{*}}{\mu \eps}\wedge n\rb + \frac{1}{\mu}\lb\frac{\H^{*}}{\mu \eps}\wedge n\rb^{\frac{2}{3}}\right\}\log \frac{\Delta_{f}}{\eps}\rb\]
\end{proof}

\end{document}